\documentclass[12pt,reqno]{amsart}   	
\usepackage[letterpaper, margin=1in]{geometry}
\usepackage{graphicx}				
								
\usepackage{amssymb}
\usepackage{mathtools,amsmath}
\usepackage{amsthm}
\usepackage{amssymb}
\usepackage{mathrsfs}
\usepackage{epstopdf}
\usepackage{color}
\usepackage{enumerate}
\usepackage[pagebackref,colorlinks,citecolor=blue,linkcolor=magenta]{hyperref}
\usepackage[capitalise]{cleveref}
\usepackage{multirow}
\usepackage{booktabs}

\usepackage[linesnumbered,ruled]{algorithm2e}
\RequirePackage{amsthm,amsmath,amsfonts,amssymb}
\usepackage[utf8]{inputenc}
\usepackage{chngcntr}
\usepackage{float}

\theoremstyle{plain}
\newtheorem{prop}{Proposition}[section]
\newtheorem{thm}[prop]{Theorem}
\newtheorem{cor}[prop]{Corollary}
\newtheorem{lem}[prop]{Lemma}
\newtheorem{conj}[prop]{Conjecture}
\newtheorem*{thm*}{Theorem}

\theoremstyle{definition}
\newtheorem{dfn}[prop]{Definition}
\newtheorem{rem}[prop]{Remark}
\newtheorem{example}[prop]{Example}
\newtheorem{lab}[prop]{}

\renewcommand{\iff}{\Leftrightarrow}

\newcommand{\C}{{\mathbb{C}}}
\renewcommand{\P}{{\mathbb{P}}}
\newcommand{\R}{{\mathbb{R}}}
\newcommand{\N}{{\mathbb{N}}}
\newcommand{\Z}{{\mathbb{Z}}}

\DeclareMathOperator{\im}{im}

\DeclareMathOperator{\interior}{int}

\DeclareMathOperator{\gl}{GL}

\DeclareMathOperator{\gram}{Gram}

\DeclareMathOperator{\spn}{span}
\DeclareMathOperator{\codim}{codim}
\DeclareMathOperator{\initial}{in}
\DeclareMathOperator{\gin}{gin}

\DeclareMathOperator{\orth}{O} 

\DeclareMathOperator{\p}{p} 

\newcommand{\Rx}{\mathbb{R}[\ul{x}]}

\newcommand{\vc}{\mathcal{V}}

\renewcommand{\setminus}{\smallsetminus}
\newcommand{\ol}{\overline}
\newcommand{\ul}{\underline}

\newcommand{\all}{\forall\,}

\renewcommand{\subset}{\subseteq}
\renewcommand{\supset}{\supseteq}

\newcommand{\mm}{\mathfrak{m}}

\newcommand{\sz}{\vc}
\newcommand{\sy}[1]{\mathcal{S}_2 #1}
\newcommand{\id}[1]{\langle #1 \rangle}

\newcommand{\bpf}{\bp-free}
\newcommand{\bp}{base-point}
\newcommand{\gs}{Gram spectrahedron}
\newcommand{\gsa}{Gram spectrahedra}
\newcommand{\cc}{change of coordinates}
\newcommand{\nonsing}{non\--singular}

\renewcommand{\emptyset}{\varnothing}
\renewcommand{\setminus}{\smallsetminus}
\renewcommand{\epsilon}{\varepsilon}
\renewcommand{\theta}{\vartheta}

\newcommand{\todfn}[1]{\textit{#1}}

\newcommand{\ra}[1]{\renewcommand{\arraystretch}{#1}}

\newcommand{\bin}[2]{\textstyle\binom{#1}{#2}}

\author[Julian Vill]{Julian Vill}
\address{	
	OvGU Magdeburg, FMA-IAG,
	D-39106 Magdeburg, Germany
}
\email{julian.vill@ovgu.de}

\title{Bounds on Hilbert functions with application to convexity}

\begin{document}

\begin{abstract}
Given a subspace $U\subset\C[x_1,\dots,x_n]_d$ we consider the closure of the image of the rational map $\P^{n-1}\dashrightarrow\P^{\dim U-1}$ given by $U$. Its coordinate ring is isomorphic to $\bigoplus_{i\ge 0} U^i$ where $U^i$ is the degree $i$ component. We consider the Hilbert function of this algebra in the case where $U$ contains a regular sequence, equivalently the map is a morphism, and find lower bounds for the dimension of the degree 2 component.
We apply our bounds to study the boundary structure of certain convex sets, called Gram spectrahedra, which are linked to sum of squares representations of non-negative polynomials.
\end{abstract}

\maketitle

\section{Introduction}

A well-studied property of ideals in standard-graded polynomial rings is their Hilbert function. A natural question to ask therefore is which sequences of integers can appear as Hilbert functions of such ideals. Macaulay answered this question and showed that whenever there exists an ideal having a certain Hilbert function, there also exists a lex-ideal with the same Hilbert function.
In particular, this implies that it suffices to consider monomial ideals.

Later, Eisenbud, Green and Harris \cite{egh1993} asked a more refined question. Given an ideal $I$ in the polynomial ring over $\C$ containing a regular sequence of degrees $d_1,\dots,d_s$. Can we find a monomial ideal containing $x_i^{d_i}$ ($i=1,\dots,s$) with the same Hilbert function. This question turned out to be very difficult and currently only several special cases have been proven.

We are interested in a similar refinement in a different situation.
Let $n,\,d\in\N$, $n,\,d\ge 2$. We denote by $A: =\C[x_1,\dots,x_n]$ the polynomial ring over $\C$. We also write $A(n)$ to emphasize the number of variables but usually omit the $n$ in the notation. For the polynomial ring over the reals we write $\Rx:=\R[x_1,\dots,x_n]$, $\ul x=(x_1,\dots,x_n)$. 

Given a positive integer $r\le \dim A_d$ and a subspace $U\subset A_d$ of dimension $r$, one may consider the rational map $\P^{n-1}\dashrightarrow\P^{r-1}$ defined by $U$. The coordinate ring of the closure of the image of this map is then given as the subalgebra $\C[U]$ defined by $U$. It carries a natural $\Z_+$-grading with the degree $i$ component being $U^i$, the subspace spanned by all homogeneous polynomials $\prod_{j=1}^i f_j$ with $f_j\in U$.

In \cite{bc2018} the authors ask the following question: For given $r,i,n,d$ what is the minimum of $\{\dim U^i\colon U\subset A_d,\, \dim U=r\}$. They show that it is attained by a strongly stable subspace. These are monomial subspaces, hence the minimum can be computed in small cases as one only has to check finitely many cases.

We are interested in the analogue of the EGH conjecture in this situation. Namely, determine the minimum if we additionally require the subspace $U$ to be \bpf, i.e. there is no point $\xi\in\P^{n-1}$ such that for every $f\in U\colon f(\xi)=0$. This is also equivalent to requiring the rational map defined by $U$ to be a morphism rather than a rational map.

As can be expected from the EGH conjecture this makes the problem even harder. Especially, it is no longer clear if it suffices to consider monomial subspaces.

In this paper, we are exclusively interested in the degree 2 component $U^2$. It is easy to see that the maximal dimension is attained by a generic subspace, i.e. being generic in the Grassmannian of $r$-planes in affine $(\dim A_d)$-space. One might expect that for a generic subspace the multiplication map $\sy U\to U^2$ from the second symmetric power is either injective or surjective. However, as shown in \cite[Prop 2.8.]{bc2018} this is not true in general.

Regarding the minimal possible dimension, if we do not require the subspace to be \bpf, we know we may take a strongly stable subspace. However, every strongly stable subspace of codimension at least 1 has a \bp\ which makes it useless for our purpose.

We note that in general, even in small cases it is a hard task to determine this value. We do for example not know the precise value already for subspaces of codimension 2 in $\R[x,y]_5$.

It turns out to be more convenient to consider the codimension of such subspaces instead of the dimension. Therefore, we are now interested in the maximum of $\{\codim U^2\colon U\subset A_d,\, \codim U=k,\, U \text{ \bpf}\}$. Our main result is the following bound for subspaces of small codimension.

\begin{thm*}[\cref{thm:main_bound}]
Let $k\le d-1$. Then for every $n\ge 2$ and every \bpf\ subspace $U\subset \C[x_1,\dots,x_n]_d$ of codimension $k$ we have 
\[
\codim U^2\le k^2+\bin{k+2}{3} = \frac{1}{6}(k^3+9k^2+2k).
\]
\end{thm*}

We want to emphasize that this bound is independent of the number of variables $n$ which heavily contrasts the general case where we allow the subspace to have \bp s. If for example $U$ has codimension 1, then the unique strongly stable subspace of codimension 1 satisfies $\codim U^2=n$.

Lastly, we apply our results to study dimensions of faces of Gram spectrahedra. Let $f\in\Rx_{2d}$ be a sum of squares, i.e. there exist $f_1,\dots,f_r\in \Rx_d$ such that $f=\sum_{i=1}^r f_i^2$. In general, for one fixed $f$ there are infinitely many such representations, both of different lengths and of the same length.
Moreover, on representations of length $r$ there is an action of the orthogonal group $\orth(r)$. Modulo this operation, we get a compact, convex set $\gram(f)$, called the Gram spectrahedron of $f$, parametrizing all sum of squares representations of $f$. It is a subset of the cone of real symmetric positive semidefinite matrices of size $\dim A_d$. As a closed, convex set its boundary is the union of faces which in this case are given by intersecting $\gram(f)$ with hyperplanes that only intersect the boundary of $\gram(f)$.
To every face $F\subset\gram(f)$ we can associate a subspace $U\subset\Rx_d$. If a relative interior point of $F$ is given by a matrix $G$ this subspace is simply the image of this $\dim A_d\times\dim A_d$-matrix embedded into $\Rx_d$. Especially, the dimension of the subspace $U$ is exactly the rank of the matrix $G$. By a result of Scheiderer \cite{scheiderer2018} the dimension of the face $F$, i.e. the dimension of its affine hull, is given by the formula
\begin{equation}
\label{eq:dim_face}
\dim F=\bin{\dim U+1}{2}-\dim U^2.
\end{equation}
Therefore, if considering faces whose relative interior points are matrices of some fixed rank $r$, the dimension of the face only depends on $\dim U^2$.
If in the representation $f=\sum_{i=1}^r f_i^2$ all forms $f_i$ vanish at a point $\xi\in\P^{n-1}$, the form $f$ vanishes to order 2 at this point which shows that $f$ is singular. 
Hence, if we are interested in \gsa\ of non-singular forms we need to study subspaces that are \bpf. Put differently, comparing general subspaces to \bpf\ subspaces corresponds to comparing \gsa\ of singular forms to \gsa\ of non-singular forms.

Our bound above immediately gives rise to an upper bound on the dimension of such faces. As it is independent of the number of variables, we see that dimensions of faces of \gsa\ of singular forms and non-singular form may be arbitrarily far apart if we allow the number of variables to grow. 

Next we explain the structure of the paper as well as the methods used in the proofs. 
In \cref{sec:Preliminaries} we recall some results about strongly stable subspaces and the bounds for $\codim U^2$ found by Boij and Conca \cite{bc2018}. At the end we introduce well-known theorems by Macaulay, Gotzmann and Green concerning the growth of Hilbert functions. \cref{sec:A first upper bound for base-point-free subspaces} contains a first easy upper bound whereas in \cref{sec:Subspaces of codimension 1 and 2} we are concerned with subspaces of codimension 1 and 2. Especially the methods used in the codimension 2 case are used again later on. After this we start proving the general bound which is split into \cref{sec:Reducing the degree} to \cref{sec:determine_bound}. In the first of the sections we study the dependence of $m(n,d,k)=\max\{\codim U^2\colon U\subset A_d\text{ subspace},\, \codim U=k\}$ on the degree $d$. After that, we analyse how to change the number of variables and end by determining the number $m(k,k,k)$. In the last section, \cref{sec:gram_spec}, we apply our results to find bounds on the dimension of faces of Gram spectrahedra.

\section{Preliminaries}
\label{sec:Preliminaries}

In this first section, we recall upper bounds for $\codim U^2$ proven by Boij and Conca \cite{bc2018}. The subspaces realizing the bounds are monomial subspaces and may even be taken to be strongly stable. Strongly stable subspaces and generic initial ideals form one of our main tools for the rest of the paper.

For most results, the monomial order does not play a role. For simplicity, if not stated otherwise, we always work with the lex-ordering. Whenever we consider the orthogonal complement of a subspace, we always work with the apolarity pairing.

\begin{dfn}
Let $U\subset A_d$ be a monomial subspace. $U$ is called \todfn{strongly stable} if for every $1 \le i \le n$ the following holds: For every monomial $M\in U$ such that $x_i|M$ and every $i<j\le n$, the monomial $x_j\frac{M}{x_i}$ is contained in $U$.
\end{dfn}

For every monomial ordering $\succeq$ where $x_1<\dots<x_n$ and every monomial $M$ such that $x_i|M$ we have $x_j\frac{M}{x_i}\succeq M$ for every $j>i$.
 
\begin{rem}
Another way of thinking about strongly stable subspaces is via their complements. If $U\subset A_d$ is strongly stable and $W=U^\perp$, then for every $1\le i\le n$ the following holds: For every monomial $M\in W$ such that $x_i|M$ and every $1\le j<i$, the monomial $x_j\frac{M}{x_i}$ is contained in $W$, i.e. the inequality sign is reversed.
\end{rem}

The following statements are all immediate to check from the definition.

\begin{lem}
\label{lem:ss_small_codim}
Let $U\subset A_d$ be a strongly stable subspace of codimension $k$.
\begin{enumerate}
\item The subspace $V:=x_1U\oplus \C[x_2,\dots,x_n]_{d+1}$ is also strongly stable.
\item If $\codim U\le d$, every monomial in $U^\perp$ is divisible by $x_1^s$ with $s:=d-k+1$.
\item If $\codim U\le n$ then every monomial in $U^\perp$ is contained in $A(k)_d=\C[x_1,\dots,x_k]_d$.
\end{enumerate}
\end{lem}

The main reason why strongly stable subspaces are particularly useful are the following propositions.

\begin{prop}[{\cite[Theorem 15.18]{eisenbud1995}}]
\label{prop:gin_existence}
Let $I\subset A$ be a homogeneous ideal. There exists a Zariski-open subset $V\subset \gl_n(\C)$ such that for every $G_1,G_2\in V$ the initial ideals satisfy $\initial(G_1I)=\initial(G_2I)$ where $G_iI: = \{p(G_i^{-1}x)\colon p\in I\}$ is the ideal $I$ is mapped to by the coordinate change $G_i$ for $i=1,2$.
\end{prop}

\begin{dfn}
Let $I\subset A$ be a homogeneous ideal and $G\in V$ as in \cref{prop:gin_existence} then
\[
\gin(I): = \initial(GI)
\]
is called the \todfn{generic initial ideal} of $I$.
\end{dfn}

Generic initial ideals have already been used by Hartshorne in 1966 (\cite{hartshorne1966}) to show the connectedness of Hilbert schemes, and later on to get hold of invariants of projective varieties.
The first systematic study of generic initial ideals in characteristic 0 was done by Galligo in \cite{galligo1974}.

\begin{prop}[{\cite[Theorem 15.20, 15.23]{eisenbud1995}}]
\label{prop:gin_is_stongly_stable}
Let $I\subset A$ be a homogeneous ideal, then for every $s \ge 0$ the vector space $\gin(I)_s$ is strongly stable.
\end{prop}

The main idea is the following easy observation.

\begin{lem}
\label{lem:powers_and_initial}
Let $I\subset A$ be a homogeneous ideal then $\initial(I)^2\subset \initial(I^2)$.
\end{lem}

\begin{rem}
In general we have $\initial(U)^2\subsetneq\initial(U^2)$. Consider $U=\spn(x_1^2+x_2^2)^\perp\subset A_2$ with $n\ge 3$. Then $U$ is spanned by all monomials except for $x_1^2$ and $x_2^2$ and by the binomial $x_1^2-x_2^2$. 
The initial ideal $\initial(U)_2$ is spanned by all monomials except for $x_1^2$, i.e. $\initial(U)_2=\spn(x_1^2)^\perp$. Now one easily checks that $\codim \initial(U^2)_{4}=\codim U^2=2$ and $\codim (\initial(U)_2)^2=n$.
\end{rem}

We want to compare the following two values.

\begin{dfn}
For $n\ge 2$ and $d,k\ge 1$ let
\begin{align*}
m(n,d,k)&=\max\{\codim U^2\colon U\subset A_d\text{ subspace},\, \codim U=k\},\\
m^0(n,d,k)&=\max\{\codim U^2\colon U\subset A_d\text{ \bpf\ subspace},\, \codim U=k\}.
\end{align*}
We say that a subspace $U\subset A_d$ of codimension $k$ \todfn{realizes} $m(n,d,k)$ (resp. $m^0(n,d,k)$) if $\codim U^2=m(n,d,k)$ (resp. $=m^0(n,d,k)$).
\end{dfn}

As the following proposition shows, the number $m(n,d,k)$ can be computed combinatorially.

\begin{prop}[{\cite[Proposition 2.2]{bc2018}}]
\label{prop:ss_is_minimal}
For all positive integers $n\ge 2,\, d\ge 2 $ and $k$, there exists a strongly stable subspace $U\subset A_d$ of codimension $k$ such that
\[
m(n,d,k) = \codim U^2.
\]
\end{prop}

\begin{rem}
\label{rem:HF_ss_subspace}
In fact, any strongly stable subspaces $U\subset A_d$ of codimension $k\le d$ is the space of all forms of degree $d$ vanishing at some $k$ points (counted with multiplicity).
Or equivalently, the Hilbert function $t\mapsto\dim\, (A/\id{U})_t$ is equal to $k$ for any $t\ge d$.
However, it is not clear in general which configuration of $k$ points realizes $m(n,d,k)$.
\end{rem}

\begin{rem}
\label{rem:example_ss}
For small $n,d,k$ this is a list of $m(n,d,k)$ for $n=2,3,4,5,6$. This has been calculated using {\ttfamily SAGE} \cite{sage} by first finding all strongly stable subspaces of some fixed codimension and then finding the maximum of all $\codim U^2$.

\begin{table*}[ht]\centering
\ra{1.3}
\setlength\tabcolsep{4.5pt}
\begin{tabular}{@{}rrrrrrrrrrcrrrrrrrrr@{}}\toprule
&\multicolumn{9}{c}{$n=3$} & \phantom{abc} & \multicolumn{8}{c}{$n=4$}\\
\midrule
$d=$ && $2$ & $3$ & $4$ & $5$ & $6$ & $7$ & $8$ & $9$ && $2$ & $3$ & $4$ & $5$ & $6$ & $7$ & $8$ & $9$\\ \midrule
$\codim U=$\\
$1$ && $3$ & $3$ & $3$ & $3$ & $3$ & $3$ & $3$ & $3$ && $4$ & $4$ & $4$ & $4$ & $4$ & $4$ & $4$ & $4$\\
$2$ && $6$ & $6$ & $6$ & $6$ & $6$ & $6$ & $6$ & $6$ && $8$ & $8$ & $8$ & $8$ & $8$ & $8$ & $8$ & $8$\\
$3$ && $10$ & $10$ & $10$ & $10$ & $10$ & $10$ & $10$ & $10$ && $13$ & $13$ & $13$ & $13$ & $13$ & $13$ & $13$ & $13$\\
$4$ && $12$ & $13$ & $13$ & $13$ & $13$ & $13$ & $13$ & $13$ && $20$ & $20$ & $20$ & $20$ & $20$ & $20$ & $20$ & $20$\\
$5$ && $14$ & $16$ & $17$ & $16$ & $16$ & $16$ & $16$ & $16$ && $23$ & $24$ & $25$ & $24$ & $24$ & $24$ & $24$ & $24$\\
$6$ && $-$ & $21$ & $21$ & $21$ & $21$ & $21$ & $21$ & $21$ && $26$ & $29$ & $29$ & $31$ & $28$ & $28$ & $28$ & $28$\\
$7$ && $-$ & $23$ & $24$ & $24$ & $25$ & $24$ & $24$ & $24$ && $30$ & $35$ & $35$ & $35$ & $37$ & $35$ & $35$ & $35$\\
$8$ && $-$ & $25$ & $27$ & $27$ & $28$ & $29$ & $27$ & $27$ && $32$ & $39$ & $40$ & $41$ & $41$ & $43$ & $40$ & $40$\\
$9$ && $-$ & $27$ & $30$ & $31$ & $31$ & $32$ & $33$ & $31$ && $34$ & $45$ & $45$ & $45$ & $47$ & $47$ & $49$ & $45$\\
\bottomrule
\bottomrule
&\multicolumn{9}{c}{$n=5$} & \phantom{abc} & \multicolumn{8}{c}{$n=6$}\\
\midrule
$d=$ && $2$ & $3$ & $4$ & $5$ & $6$ & $7$ & $8$ & $9$ && $2$ & $3$ & $4$ & $5$ & $6$ & $7$ & $8$ & $9$\\ \midrule
$\codim U=$\\
$1$ && $5$ & $5$ & $5$ & $5$ & $5$ & $5$ & $5$ & $5$ && $6$ & $6$ & $6$ & $6$ & $6$ & $6$ & $6$ & $6$\\
$2$ && $10$ & $10$ & $10$ & $10$ & $10$ & $10$ & $10$ & $10$ && $12$ & $12$ & $12$ & $12$ & $12$ & $12$ & $12$ & $12$\\
$3$ && $17$ & $16$ & $16$ & $16$ & $16$ & $16$ & $16$ & $16$ && $21$ & $19$ & $19$ & $19$ & $19$ & $19$ & $19$ & $19$\\
$4$ && $24$ & $25$ & $24$ & $24$ & $24$ & $24$ & $24$ & $24$ && $28$ & $31$ & $28$ & $28$ & $28$ & $28$ & $28$ & $28$\\
$5$ && $35$ & $35$ & $35$ & $35$ & $35$ & $35$ & $35$ & $35$ && $40$ & $40$ & $41$ & $40$ & $40$ & $40$ & $40$ & $40$\\
$6$ && $39$ & $40$ & $40$ & $41$ & $40$ & $40$ & $40$ & $40$ && $56$ & $56$ & $56$ & $56$ & $56$ & $56$ & $56$ & $56$\\
$7$ && $43$ & $47$ & $45$ & $46$ & $49$ & $45$ & $45$ & $45$ && $61$ & $62$ & $62$ & $62$ & $62$ & $62$ & $62$ & $62$\\
$8$ && $48$ & $54$ & $55$ & $54$ & $54$ & $57$ & $54$ & $54$ && $66$ & $71$ & $68$ & $68$ & $68$ & $71$ & $68$ & $68$\\
$9$ && $55$ & $60$ & $60$ & $63$ & $61$ & $62$ & $65$ & $59$ && $73$ & $79$ & $81$ & $79$ & $79$ & $79$ & $81$ & $79$\\
\bottomrule
\end{tabular}
\end{table*}

We have $m(n,d,1)=n$ and $m(n,d,2)=2n$ in every case shown in the table. This can also easily be checked to hold in general.

Moreover, we see that in these examples $m(n,d,k)=m(n,k,k)$ for any $d\ge k$, i.e. values are constant on the right side of the diagonal. This is always true as we show later (see \cref{cor:ss_bound_independent_of_d}).
\end{rem}

\begin{rem}
\label{rem:asymptotic_behavior}
We now discuss the asymptotic behavior of $m(n,d,k)$. As we have mentioned, for fixed $n$ and $k$, the number $m(n,d,k)$ stabilizes for large $d$. To be more precise, we have $m(n,d,k)=m(n,k,k)$ for every $d\ge k$.

Determining exactly the growth of $m(n,d,k)$ for increasing $n$ or $k$ seems to be a rather difficult combinatorial problem. However, it is clear that increasing $n$ or $k$ while fixing the other value and the degree, results in larger values for $m(n,d,k)$. More precisely, we do get a lower bound for the growth using the Alexander-Hirschowitz Theorem.

Let $X\subset\P^{n-1}$ be a set of $k$ general points, in the sense of the Alexander-Hirschowitz Theorem. Assume that $d$ is large enough, then the subspace $U$ of forms in $A_d$ vanishing on $X$ has codimension $k$. By the Alexander-Hirschowitz Theorem, the space $V$ of all forms of degree $2d$ vanishing to order at least 2 at every point of $X$ has codimension $kn$.
Since $U^2\subset V$, it follows that $\codim U^2\ge kn$, especially $m(n,d,k)\ge kn$. And therefore also $m(n,d',k)\ge kn$ for every $d'\ge k$.
\end{rem}

\begin{rem}
Let $U\subset A_d$ be a strongly stable subspace such that $U\neq A_d$. If $x_1^d\in U$, we see from the definition that $U=A_d$. Therefore, $x_1^d\in U^\perp$ which reveals a \bp\ of $U$.
This shows that if $f\in U^2$, then $f$ is singular at the point $(1 : 0 :\dots : 0)$.
\end{rem}

For later reference, we now consider \bpf\ monomial subspaces $U$ and find bounds for $\codim U^2$. This will be needed later on as we will reduce to the monomial case.

\begin{lem}
\label{lem:monomial_codim_1}
Let $d\ge 2$ and let $U\subset A_{d}$ be a \bpf, monomial subspace of codimension 1. Then the following hold: 
\begin{enumerate}
\item If $d=2$ then $\codim U^2=2$,
\item if $d\ge 3$ then $\codim U^2\in\{0,1\}$.
\end{enumerate}
\end{lem}
\begin{proof}
Let $U^\perp=\spn(M)$ for some monomial $M\in A_d$. Up to permutation of the variables there are only two monomials in $A_{2d}$ that have only one decomposition into a product of monomials of degree $d$, those are $x_1^{2d}$ and $x_1^{2d-1}x_2$.

Let $T\in A_{2d}$ be any monomial that is not $x_1^{2d}$ or $x_1^{2d-1}x_2$ (after permutation of the variables). Then there are two decompositions into a product of two monomials of degree $d$. Especially, one of the decompositions does not use the monomial $M$, hence $T\in U^2$.

The decompositions of the two monomials above are $x_1^{2d}=(x_1^d)(x_1^d)$ and $x_1^{2d-1}x_2=x_1^d(x_1^{d-1}x_2)$. 
Therefore, both are not contained in $U^2$ if and only if $M=x_1^d$, and only $x_1^{2d-1}x_2$ is not contained in $U^2$ if and only if $M=x_1^{d-1}x_2$.
In the first case, $U$ has a \bp, in the second case the only monomial not contained in $U^2$ is $x_1^{2d-1}x_2$ if $d\ge 3$ and thus $\codim U^2=1$. 

If $d=2$ and $U$ is \bpf, $U^\perp$ is spanned by $x_1x_2$ (after permutation of the variables). We easily check that $U^2$ contains every monomial of degree 4 except for $x_1^3x_2$ and $x_1x_2^3$. Hence $\codim U^2=2$.
\end{proof}

\begin{lem}
\label{lem:monomial_subspaces_of_codimension_2}
Let $U\subset A_d$ be a \bpf, monomial subspace of codimension 2. Then the following hold: 
\begin{enumerate}
\item $\codim U^2\le 6$ if $d=2$,
\item $\codim U^2\le 4$ if $d\in\{3,4\}$, 
\item $\codim U^2\le 2$ if $d\ge 5$.
\end{enumerate}
Moreover the bound is tight for $d\le 4$.
\end{lem}
\begin{proof}
The proof works the same as for \cref{lem:monomial_codim_1} by considering all monomials with at most two distinct decompositions.
\end{proof}

For the following sections, we need some knowledge about the Hilbert functions of ideals generated by subspaces. We introduce theorems of Macaulay and Gotzmann concerning Hilbert functions and Green's Hyperplane Restriction Theorem for later reference.

\begin{dfn}
Let $a,d\in\N$, then $a$ can be uniquely written in the form
\[
a=\bin{k(d)}{d}+\bin{k(d-1)}{d-1}+\dots+\bin{k(1)}{1},
\]
where $k(d)>k(d-1)>\dots>k(1)\ge 0$, called the \todfn{$d$-th Macaulay representation} of $a$ (see \cite[Lemma 4.2.6.]{bh1998}). For any integers $s,t\in\Z$ define
\[
a_{(d)}|^s_t: = \bin{k(d)+s}{d+t}+\bin{k(d-1)+s}{d-1+t}+\dots+\bin{k(1)+s}{1+t}.
\]
Furthermore for $a<b$ we define $\bin{a}{b}=0$.
\end{dfn}

\begin{thm}[Macaulay's Theorem, {\cite[Corollary C.7.]{ikl1999}}, {\cite[Theorem 4.2.10]{bh1998}}]
\label{thm:macaulay}
Let $I\subset A$ be a homogeneous ideal and let $H=(h_i)_{i\ge 0}$ be the Hilbert function of $I$. 
Then 
\begin{enumerate}
\item $h_{i+1}\le (h_i)_{(i)}|^{1}_{1}$ for every $i\ge 0$, and
\item if there exists $j\in\N$ such that $j\ge h_j$, then $h_i\ge h_{i+1}$ for every $i\ge j$.
\end{enumerate}
\end{thm}

In fact, Macaulay showed in 1927 \cite{macaulay1927} that whenever we have a sequence $H=(h_i)_{i\ge 0}$ which satisfies property (i) in \cref{thm:macaulay}, there exists $n\ge 2$ and a homogeneous ideal $I\subset A$ such that the Hilbert function of $I$ is exactly $H$. This ideal $I$ can even be chosen to be monomial.

\begin{thm}[Gotzmann's Persistence Theorem, {\cite[Corollary C.17.]{ikl1999}}, {\cite[Theorem 2.6]{ams2018}}]
\label{thm:gotzmannpersistence}
Let $d \ge 0$ be an integer and let $I$ be a homogeneous ideal that is generated in degrees at most $d$ $(I=\id{I_{\le d}})$. Denote by $H=(h_i)_{i\ge 0}$ the Hilbert function of $I$. 
If $h_{d+1}=(h_d)_{(d)}|^{1}_{1}$, then 
$h_{d+l}=(h_d)_{(d)}|^l_l$ for all $l\ge 1$.
\end{thm}

By Macaulay's Theorem $h_{d+1}\le (h_d)_{(d)}|^{1}_{1}$. Therefore, Gotzmann's Theorem determines the complete Hilbert function whenever we have maximal growth from some degree $d$ to the next degree $d+1$. Namely, the growth is maximal for all following degrees as well. For ideals generated by subspaces, this has the following meaning.

\begin{cor}
\label{cor:macaulay_gotzmann}
Let $U\subset A_d$ be a subspace of codimension $k\le d$ and let $H=(h_i),\ h_i: = h_{\id{U}}(i)$ be the Hilbert function of $\id{U}$. Then
\begin{enumerate}
\item $h_{d+1}=\codim A_1U\le k$ and 
\item if $h_{d+1}=k$, then $h_{d+i}=\codim A_iU=k$ for all $i\ge 1$.
\end{enumerate}
In case (ii) $\sz(U)\neq\emptyset$ is finite.
\end{cor}
\begin{proof}
(i): We have $\codim U =h_d =k \le d$ and therefore $\codim A_1U=h_{d+1}\le h_d= k$ by \cref{thm:macaulay} (ii). 

(ii): We first note that the $d$-th Macaulay representation of $h_d$ is given by 
\[
h_d=\bin{d}{d}+\dots+\bin{d-k+1}{d-k+1}=\sum_{i=0}^{k-1} \bin{d-i}{d-i}
\]
and therefore $(h_d)_{(d)}|^{1}_{1}=h_d=k$. Hence, the assumption $h_{d+1}=k=(h_d)_{(d)}|^{1}_{1}$ allows us to use \cref{thm:gotzmannpersistence} from which we get
\[
h_{d+i}=(h_d)_{(d)}|^{i}_{i}=\bin{d+i}{d+i}+\dots+\bin{d-h_d+1+i}{d-h_d+1+i}=k
\]
for every $i\ge 1$.

In (ii) the Hilbert polynomial is the constant polynomial $k$, hence $\sz(U)$ is non-empty and finite. 
\end{proof}

\begin{cor}
\label{cor:degree2d-1}
Let $U\subset A_d$ be a \bpf\ subspace with $\codim U=k\le d$. Then $h_{\id{U}}(2d-1)\le 1$. If $k<d$ then $h_{\id{U}}(2d-1)=0$.
\end{cor}
\begin{proof}
The Hilbert function of $\id{U}$ has to be smaller than $(\dots,k,k-1,k-2,\dots,1,0)$ (dimension dropping by at least 1 in every degree): indeed, if we had equality in any two consecutive degrees $s$ and $s+1$ with $s\ge d$ such that $h_s\neq 0$, it follows from \cref{cor:macaulay_gotzmann} (ii) that $\vc(U)\neq \emptyset$.
Therefore, we get the inequality on the degree $2d-1$ component of $A/\id{U}$.
\end{proof}

For the degree $2d$ component, there is a stronger result due to Blekherman using Cayley-Bacharach duality.

\begin{thm}[{\cite[Theorem 2.5.]{blekherman2015}}]
\label{thm:degree2d}
Let $n\ge 3,\, d\ge 3$ and let $U\subset A_d$ be a \bpf\ subspace. 
If $\codim U<3d-2$, then $UA_d=A_{2d}$. If $n\ge 4,\, d=2$ and $\codim U<5$, then $UA_2=A_4$.
\end{thm}

\begin{rem}
If we would use the same argument as in \cref{cor:degree2d-1}, we only get $h_{\id{U}}(2d)=0$ if $\codim U\le d$, instead of whenever $\codim U < 3d-2$ ($<5$ if $d=2$).

This also shows that the bound $\codim U\le d$ in \cref{cor:degree2d-1} is far off from being necessary to obtain $h_{\id{U}}(2d-1)=0$ in general.

However, \cref{thm:degree2d} only tells us something about $UA_d$ and not about $UA_{d-1}$ and the proof does not easily generalize to other degrees but is very specific to the degree $2d$ component $UA_d$. 
\end{rem}

\begin{dfn}
Let $I\subset A$ be a homogeneous ideal and $p\in A_s$ for some $s\ge 1$. We define the \todfn{ideal quotient}
\[
(I:p): = \bigoplus_{l\ge 0} (I:p)_l
\]
where
\[
(I:p)_l: = \{q\in A_l\colon pq\in I\}\subset A_l
\]
for every $l\ge 0$. If $U\subset A_d$ is a subspace, we write $(U:p): = (\id{U}:p)_{d-s}\subset A_{d-s}$.
\end{dfn}

\begin{lab}
\label{lab:setup_green}
We consider the following setup. Let $I\subset A$ be a homogeneous ideal and $l\in A_1$ a linear form. We have the graded exact sequence
\[
0 \to A/(I:l)(-1) \stackrel{\cdot l}{\to} A/I \to A/\id{I,l} \to 0.
\]
Let $h_i  = \dim (A/I)_i$ and $c_i  = \dim (A/(I,l))_i$.
\end{lab}

In this situation, we have the following theorem due to Green.

\begin{thm}[Green's Hyperplane Restriction Theorem, {\cite[Theorem 1]{green1989}}]
\label{thm:green_thm}
For any $d\ge 0$ and a generic linear form $l\in A_1$ we have
\[
c_d \le (h_d)_{(d)}|^{-1}_0.
\]
\end{thm}

This can either be seen as a lower bound for $\dim \id{I,l}_d$ or equivalently as an upper bound for $\dim (I:l)_{d-1}$ which tells us how many elements in $I$ are divisible by $l$.

Notation-wise this means that if $h_d=\bin{k(d)}{d}+\bin{k(d-1)}{d-1}+\dots+\bin{k(1)}{1}$, then 
\[
c_d\le \bin{k(d)-1}{d}+\bin{k(d-1)-1}{d-1}+\dots+\bin{k(1)-1}{1}.\]

\begin{example}
Let $U\subset A_d$ be a subspace of codimension 1 and let $l\in A_1$ be a generic linear form. This means $h_d=A_d/U=1=\bin{d}{d}$. Therefore, \cref{thm:green_thm} shows
\[
c_d\le \bin{d-1}{d}=0.
\]
On the one hand, this means $\id{U,l}_d=A_d$, and on the other hand 
\[
\codim (U:l)_{d-1}=\codim U-\dim \id{I,l}_d=1.
\]
I.e. the subspace $(U:l)_{d-1}$ also has codimension 1.
\end{example}

\section{A first upper bound for base-point-free subspaces}
\label{sec:A first upper bound for base-point-free subspaces}

One way to get a good upper bound for \bpf\ subspaces of small dimension is the following.

\begin{prop}
\label{prop:small_subspaces}
Let $U\subset A_d$ be a \bpf\ subspace of dimension $r$. Then 
\[
\dim U^2 \ge n r-\bin{n}{2}.
\]
\end{prop}
\begin{proof}
Since $U$ is \bpf\ it follows that $\dim U\ge n$ and $U$ contains a regular sequence $p_1,\dots,p_n$. Consider the map
\[
\prod_{i=1}^n U\to A_{2d},\quad (q_1,\dots,q_n)\mapsto \sum_{i=1}^n p_iq_i.
\]
Since $p_1,\dots,p_n$ is a regular sequence, the syzygies are the obvious ones, namely the kernel is spanned by the vectors $(0,\dots,0,p_j,0,\dots,0,-p_i,0,\dots,0)$ with $i<j$ and $p_j$ at position $i$ in the vector and $-p_i$ at position $j$. There are exactly $\bin{n}{2}$ of those vectors, hence the image $\spn(p_1,\dots,p_r)U$ has dimension $n\cdot\dim U -\bin{n}{2}$ and the image is contained in $U^2$.
\end{proof}

\begin{cor}
For any $n,d\ge 1$, $1\le k\le \dim A_d$ we have
\[
m^0(n,d,k)\le \bin{n-1+2d}{n-1}+\bin{n}{2}+nk-n\bin{n-1+d}{n-1}
\]
\end{cor}

\begin{example}
\label{ex:reg_seq_tight}
The easiest example where this bound is tight is $U=\spn(x_1^d,\dots,x_n^d)$. Since $U$ is spanned by a regular sequence the condition $\spn(p_1,\dots,p_n)U=U^2$ is certainly true.

In general, one should expect this bound to be good whenever the dimension of the subspace is close to $n$ and rather bad whenever $\codim U$ is small.

However, the bound can also be tight or almost tight even for larger subspaces. Indeed, in any of the following cases, there exists a \bpf\ subspace $U\subset A_d$ of dimension $r$ such that the bound in \cref{prop:small_subspaces} is tight.
\begin{enumerate}
\item $r=n$,
\item $d$ is even, $n\ge 3$ and $r=n+3$,
\item there exists $s\in\N$ such that $s\vert d$, and $r=n+s-1$.
\end{enumerate}
In the next two cases, there exists a \bpf\ subspace $U\subset A_d$ of dimension $r$ such that the bound is 1 off, i.e. $\dim U^2 = nr-\bin{n}{2} +1$.
\begin{enumerate}
\item[(iv)] $d$ is even, $n\ge 4$ and $r=n+6$,
\item[(v)] $3\vert d,\, n\ge 3$ and $r=n+7$.
\end{enumerate}
The proof is not hard but we do not include a proof here as this is not needed later. The main reason is that the second Veronese of $\P^2$ and the $s$-th Veronese of $\P^1$ ($s\in\N$) are both varieties of minimal degree and the second (resp. third) Veronese of $\P^3$ (resp. $\P^2$) is an arithmetically Cohen-Macaulay variety of almost minimal degree.

The third case is especially interesting since it shows that if the degree is large enough, there exist subspaces of any dimension in any number of variables such that the bound is tight.
\end{example}

The main downside of this bound is that for large subspaces, i.e. small codimension, this bound depends on $n$ which is not necessary as shown in \cref{thm:main_bound}.
From now on we consider only the case where the codimension of $U$ is small.

\section{Subspaces of codimension 1 and 2}
\label{sec:Subspaces of codimension 1 and 2}

We start by determining bounds for $\codim U^2$ in the cases $\codim U=1,2$.
We show that there is a uniform bound for $\codim U^2$ not depending on $n$ or $d$. This is also our main motivation for the next sections where we generalize this result to higher codimensions.
Furthermore, we show \cref{thm:reduction_number_of_variables} which is our main tool in the next sections to reduce the number of variables.

\begin{lem}
\label{lem:not_contained_in_Pn-3}
Let $U\subset A_d$ be a \bpf\ subspace and $W: = U^\perp$. If $\sz(W)$ is not contained in any linear variety of codimension 2, then there exists a \cc\ such that $\initial(U)_d$ is \bpf.

In the case $n=2$, this should be understood as $\vc(W)\neq\emptyset$, i.e. $\dim\vc(W)\in\{0,1\}$.
\end{lem}
\begin{proof}
By assumption there exist linearly independent linear forms $l_1,\dots,l_{n-1}\in A_1$ such that $l_1^d,\dots,l_{n-1}^d\in U$. After a \cc, we can assume that $x_1^d,\dots,x_{n-1}^d\in U$. Since $x_1<x_2<\dots <x_n$, it holds that $x_n^d\ge x^\alpha$ for any $\alpha=(\alpha_1,\dots,\alpha_n)\in\Z^n_+=\{(a_1,\dots,a_n)\in\Z^n\colon a_i\ge 0,\, \all i=1,\dots,n\},\ |\alpha|:=\sum_{i=1}^n\alpha_i=d$. Since $U$ is \bpf, there exists a form in $U$ such that $x_n^d$ occurs in it. Hence, $\initial(U)_d$ contains $x_1^d,\dots,x_n^d$ which shows that the subspace $\initial(U)_d$ is \bpf. 
\end{proof}

With this lemma in place, we look at subspaces of codimension 1. Firstly, we consider the simple case where our subspace does have a \bp.

\begin{lem}
\label{lem:basepoint}
If $U\subset A_d$ is a subspace of codimension 1 and $U$ has a \bp, then $\codim U^2=n$. 
\end{lem}
\begin{proof}
We can apply a \cc\ such that $U^\perp=\spn(x_1^d)$. Then $U$ is the subspace spanned by all monomials except $x_1^d$. Now we see that for every $1\le i\le n$ the monomial $x_1^{2d-1}x_i$ is not contained in $U^2$ and thus $\codim U^2=n$.
\end{proof}

\begin{prop}
\label{prop:codim1possibledim}
Let $d\ge 2$ and $U\subset A_d$ be a \bpf\ subspace of codimension 1. Then the following hold:
\begin{enumerate}
\item If $d\ge 3$, then $\codim U^2\le 1$,
\item if $d=2$, then $\codim U^2\le 2$.
\end{enumerate} 
\end{prop}
\begin{proof}
Write $W: = U^\perp=\spn(q)$ for some $q\in A_d$. No hypersurface is contained in a linear variety of codimension 2, hence by \cref{lem:not_contained_in_Pn-3} we can apply a \cc\ and assume that the subspace $\initial(U)_d$ is \bpf. Then $\dim U^2=\dim \initial (U^2)_{2d}\ge\dim\, (\initial(U)^2)_{2d}$ by \cref{lem:powers_and_initial}. It therefore suffices to check the claim for monomial subspaces which is done in \cref{lem:monomial_codim_1}.
\end{proof}

\begin{rem}
In the case $d=2$ it is even true that $\codim U^2\in\{0,2\}$, i.e. the case $\codim U^2=1$ does not occur. This can be shown by considering the orthogonal complement of $U$. It is spanned by a quadratic form $q\in A_2$ of rank at least 2. As quadratic forms are diagonalizable, one can reduce to a combinatorial situation. If the rank is equal to 2, the codimension of $U^2$ is 2, in any other case $U^2=A_{2d}$.
\end{rem}

Now we turn to the codimension 2 case. We find a bound for $\codim U^2$ by reducing either to monomial subspaces or subspaces of binary forms.

First, we show how to reduce the number of variables. The idea of the proof is the following: if $U\subset A[x_{n+1}]_d=\C[x_1,\dots,x_{n+1}]_d$ is a subspace of the form $U=x_{n+1}A[x_{n+1}]_{d-1}\oplus U'$ with $U'\subset A_d$, then $U^2=x_{n+1}^2A[x_{n+1}]_{2d-2}\oplus x_{n+1}A_{d-1}U'\oplus (U')^2$. This shows 
\[
\codim U^2=\codim (U')^2+\codim A_{d-1}U'.
\]
If $U$ does not have this nice form, we have to argue slightly more carefully using the same idea.

\begin{thm}
\label{thm:reduction_number_of_variables}
Let $U\subset A_d$ be a subspace of codimension $k$. If there exists $2\le m\le n\ (R: = A(m))$ such that $U': = U\cap R_d$ satisfies $\codim_{R_d} U'=k$, then
\[
\codim_{A_{2d}} U^2\le (n-m)\codim_{R_{2d-1}} U'R_{d-1}+\codim_{R_{2d}} (U')^2.
\]
\end{thm}
\begin{proof}
Let $\mm=\id{x_{m+1},\dots,x_n}\subset A_d$, then $\mm_d = \sum_{i=m+1}^n x_iA_{d-1}$. We write 
\[
U=U'\oplus V \oplus W
\]
with $U'\subset R_d,\ V\subset \mm_d$ and $W=\spn(p_i+q_i\colon i=1,\dots,s)$ where $p_i\in R_d$ and $0\neq q_i\in \mm_d$ for $i=1,\dots,s$. By assumption $\codim_{R_d} U'=k$ which means $U+R_d=A_d$ and thus
\begin{equation}
\label{eq:eq1}
V \oplus \spn(q_1,\dots,q_s)=\mm_d.
\end{equation}
Calculating $U^2$ we get
\[
U^2=(U')^2+(V+W)^2+U'(V+W).
\]
Since we are working with the lex-ordering (and $x_1<\dots<x_n$), any monomial of degree $d$ containing any $x_i,\ i\ge m+1$ is bigger than any monomial in $R_d$.

Firstly, fix any monomial $x^\alpha$ such that $\alpha\in\Z^{n}_+, |\alpha|=2d$ and $\sum_{j\ge m+1} \alpha_j\ge 2$, then there exist $\beta,\gamma\in\Z^{n}_+,|\beta|=|\gamma|=d$ and $x_i,x_j,\ i,j\ge m+1$ such that $x_i|x^\beta, x_j|x^\gamma$ and $x^\alpha=x^\beta x^\gamma$. Then we have $x^\beta + p_\beta,\, x^\gamma + p_\gamma\in V+W$ for some $p_\beta,p_\gamma\in R_d$. Hence
\[
x^\alpha=\initial((x^\beta + p_\beta)(x^\gamma + p_\gamma))\in\initial((V+ W)^2)_{2d}\subset\initial(U^2)_{2d}.
\]
Secondly, we have
\[
\initial(U'(V+W))\stackrel{\cref{eq:eq1}}{\supset}\initial(U'\mm_d)
=\initial\left(\bigoplus_{i=m+1}^n x_i (U'A_{d-1})\right)
\supset\initial\left(\bigoplus_{i=m+1}^n x_i (U'R_{d-1})\right).
\]
This shows that for every $i=m+1,\dots, n$ we have
\[
\mm^2A_{2d-2},\, \initial((U')^2)_{2d},\, \initial(x_iU'R_{d-1})_{2d}\subset\initial(U^2)_{2d}.
\]
Counting dimensions, we get 
\[
\codim_{A_{2d}}\initial(U^2)_{2d}\le (n-m)\codim_{R_{2d-1}} U'R_{d-1}+\codim_{R_{2d}} (U')^2.
\]
\end{proof}

\begin{rem}
\label{rem:variable_reduction_tight}
The bound is sharp whenever $U=(x_{m+1},\dots,x_n)A_{d-1}\oplus U'$ as can be seen from the comment above \cref{thm:reduction_number_of_variables}.
\end{rem}

\begin{cor}
\label{cor:reduction_number_of_variables}
If the subspaces $U,U'$ in \cref{thm:reduction_number_of_variables} are \bpf\ and $k\le d-1$, then
\[
\codim U^2\le \codim (U')^2.
\]
\end{cor}
\begin{proof}
By \cref{cor:degree2d-1} the degree $2d-1$ component of $R/\id{U'}$ has dimension 0. Therefore, the result follows from \cref{thm:reduction_number_of_variables}. 
\end{proof}

\begin{thm}
\label{thm:codimension2bounds}
Let $U\subset A_d$ be a \bpf\ subspace of codimension 2. Then the following hold:
\begin{enumerate}
\item If $d=2$, then $\codim U^2\le 6$,
\item if $d\ge 3$, then $\codim U^2\le 4$.
\end{enumerate}
For $d\le 4$ the bounds are tight.
\end{thm}
\begin{proof}
Let $W=U^\perp$. If $\sz(W)\neq\sz(l,l')$ for any two linear forms $l,l'\in A_1$, the claim follows from \cref{lem:not_contained_in_Pn-3} and \cref{lem:monomial_subspaces_of_codimension_2} with the same arguments as in the codimension 1 case (\cref{prop:codim1possibledim}) as we can reduce to \bpf\ monomial subspaces.

Otherwise we can assume after a \cc\ that $\sz(W)=\sz(x_1,x_2)$ and thus $x_3^d,\dots,x_n^d\in U,\ x_1^d,x_2^d\notin U$. 
Hence, we can write
\[
U=\spn(x^\alpha+\nu_\alpha x_1^d+\lambda_\alpha x_2^d\colon \alpha\in\Z^n_+,\, |\alpha|=d,\, \exists\, i\ge 3\colon x_i | x^\alpha)\oplus U'
\]
where $U'\subset\C[x_1,x_2]_d$ is a subspace of codimension 2. We distinguish two cases. Either (a) for all $\alpha$ we have $\nu_\alpha=\lambda_\alpha=0$, or (b) there exists $\alpha$ such that $(\nu_\alpha,\lambda_\alpha)\neq (0,0)$.

(a): Here $U$ has the form
\[
U=\spn(x_3,\dots,x_n)A_{d-1}\oplus U'.
\]
If $d=2$, this case cannot occur since $\dim U'=1$ and thus $U$ has a \bp. Hence, we can assume that $d \ge 3$. Since $U$ is \bpf\ it follows that $U'$ is \bpf\ as a subspace of $\C[x_1,x_2]_d$. 
Then $\codim U^2 \le \codim (U')^2\le 4$ by \cref{cor:reduction_number_of_variables}.

(b): Fix $\alpha\in\Z^n_+,\, |\alpha|=d$ such that $(\nu_\alpha,\lambda_\alpha)\neq (0,0)$. Consider the subspace 
\[
V: = U'\oplus \spn(\nu_\alpha x_1^d+\lambda_\alpha x_2^d)\subset\C[x_1,x_2]_d.
\]
This subspace has codimension 1 and thus $V^\perp=\spn(h)$ for some $h\in\C[x_1,x_2]_d$. Especially, there exists $l\in\C[x_1,x_2]_1$ such that $l^d\in V$, namely the one evaluating $h$ in one of its zeroes. 
Therefore, there exist $a\in\C$ and $\beta\in\Z_+^n,\, \vert\beta\vert=d,\, x^\beta\notin\C[x_1,x_2]$ such that $ax^\beta+l^d\in U$. Write $x^\beta=x_1^{\beta_1}x_2^{\beta_2}M$ with $M\in\C[x_3,\dots,x_n]$.
Let $\phi$ be a \cc\ on $\C[x_1,x_2]$ that maps $l$ to $x_2$. Then $aMg+x_2^d\in \phi(U)$ with $g\in\C[x_1,x_2]$ the image of $x_1^{\alpha_1}x_2^{\alpha_2}$ under $\phi$. Now take any monomial ordering such that $x_1>x_2>\dots >x_n$ and such that $x_2^d$ is greater than any monomial in $Mg$, for example, the ordering given in \cref{rem:block_ordering}.
Wrt this ordering $\initial(\phi(U))_d$ contains $x_1^d,\dots,x_n^d$ and is therefore \bpf: the monomials $x_3^d,\dots,x_n^d$ are contained in $U$ and therefore in $\phi(U)$ by assumption. The monomial $x_2^d$ is the initial monomial of $aMg+x_2^d$ and $x_1^d$ appears in some form in $\phi(U)$ since it is \bpf\ and by the choice of the monomial ordering $x_1^d$ is the initial monomial of that form. Now we finish as earlier, $\codim U^2=\codim \phi(U)^2\le \codim\, (\initial(\phi(U))_d)^2$ and using \cref{lem:monomial_subspaces_of_codimension_2} we get the bounds we wanted.

The bounds are tight for $d\le 4$ by \cref{lem:monomial_subspaces_of_codimension_2}.
\end{proof}

\begin{rem}
\label{rem:block_ordering}
We want to define a monomial ordering such that $x_1>x_2>\dots>x_n$ and such that $x_2^d$ is greater than any monomial of degree $d$ that is divisible by $x_i$ for any $i\in\{3,\dots,n\}$.

We consider a block ordering $\succeq_b$ on the sets $\{x_1,x_2\}$ and $\{x_3,\dots,x_n\}$ and on each set the graded-lexicographic-ordering (grlex).
Let $\alpha,\,\beta\in\Z_+^{n}$, then the grlex ordering is defined as
\[
x^\alpha \succeq_{grlex} x^\beta : |\alpha| > |\beta| \text{ or } |\alpha|=|\beta| \text{ and } x^\alpha \succeq_{lex} x^\beta
\]
where $\succeq_{lex}$ is the usual lex-ordering and the variables are ordered as $x_1>x_2>\dots>x_n$.
Then the block-ordering is defined as follows. Let $\alpha,\,\beta\in\Z_+^n$, then
\[
x^\alpha\succeq_b x^\beta :\iff x_1^{\alpha_1}x_2^{\alpha_2} \succ_{grlex} x_1^{\beta_1}x_2^{\beta_2} \text{ or } x_1^{\alpha_1}x_2^{\alpha_2}=x_1^{\beta_1}x_2^{\beta_2} \text{ and } \frac{x^\alpha}{x_1^{\alpha_1}x_2^{\alpha_2}}\succeq_{grlex} \frac{x^\beta}{x_1^{\beta_1}x_2^{\beta_2}}.
\]
Now let $x^\alpha\in A_d$ be any monomial such that $\alpha_i>0$ for some $i\in\{3,\dots,n\}$. Then $\alpha_1+\alpha_2<d$ and therefore $x_2^d\succ_b x^\alpha$.
\end{rem}

\begin{rem}
\label{rem:egh}
The idea to choose a monomial ordering such that the degree $d$ component of the initial ideal is \bpf\ is unlikely to work as easily for higher codimensions.

The lex-plus-powers conjecture (or EGH conjecture) due to Eisenbud, Green, and Harris \cite{egh1993} predicts that for any homogeneous ideal $I\subset A$ containing a regular sequence $p_1,\dots,p_n$ with $d_i: = \deg p_i$, there is also a monomial ideal containing $x_i^{a_i}$ for $i=1,\dots,n$ with the same Hilbert function as $I$.
The conjecture has only been proven in some special cases, see for example \cite{cm2008}.

This is certainly not exactly what we are after. On the one hand, we are only interested in the case where all $p_i$ are of the same degree and the ideal is generated in that degree. On the other hand, it is not enough that the monomial ideal has the same Hilbert function, since the Hilbert function of $I$ does not determine the Hilbert function of $I^2$.

This however shows that we should not expect a reduction to monomial ideals to easily work in more general cases.
\end{rem}

We have seen that in the codimension 1 case the codimension of $U^2$ can be bounded by $1$ if $d\ge 3$ (resp. $2$ if $d=2$) and in the case $\codim U=2$, the codimension can be bounded by $4$ if $d\ge 3$ (resp. $6$ if $d=2$).
We would like to generalize this to higher codimensions. It seems that the correct way to do this is to show that there is a bound for $\codim U^2$ that is not dependent on $n$ or $d$, as long as $d$ is large enough.

We have already seen how we can reduce the number of variables and therefore make our bounds independent of $n$ using \cref{thm:reduction_number_of_variables}.

In the next section, we show how to find bounds that are independent of the degree $d$.

\section{Reducing the degree}
\label{sec:Reducing the degree}

In this section, we show that for $d\ge k$ the function $d\mapsto m(n,d,k)$ is constant for every fixed $n,k$. By definition, this is equivalent to showing that for certain subspaces $U\subset A_d$ of codimension $k$ there exists a subspace $V\subset A_{d-1}$ of codimension $k$ such that $\codim U^2 = \codim V^2$ whenever $d > k$.

For the next proofs, let us recall that for any subspace $U\subset A_d$ we have the exact sequence in \ref{lab:setup_green}:
\[
0 \to A_{d-1}/(U:l) \to A_d/U \to A_d/\id{U,l}_d \to 0.
\]
Especially, if $\id{U,l}_d=A_d$ it follows that $\codim (U:l) = \codim U$.

\begin{lem}
\label{lem:green_small_k}
Let $U\subset A_d$ be a subspace of codimension $k$ and $k\le d$. Then for a generic linear form $l\in A_1$ we have $\id{U,l}_d=A_d$ and $\codim (U:l)=\codim U$.
\end{lem}
\begin{proof}
With the notation from \cref{lab:setup_green} and with $I=\id{U}$ we have
\[
h_I(d)=h_d=k=\sum_{i=0}^{k-1} \bin{d-i}{d-i}
\]
since $k\le d$.
Hence, by Green's Theorem 
\[
\dim A_d/\id{U,l}_d= c_d\le (h_d)_{(d)}\vert^{-1}_0 = \sum_{i=0}^{k-1} \bin{d-i-1}{d-i}=0
\]
which means $A_d/\id{U,l}_d=0$, and therefore the first claim follows.
The second one is immediate from the exact sequence above.
\end{proof}

\begin{thm}
\label{thm:deg_reduction}
Let $U\subset A_d$ be a subspace of codimension $k\le d$ and let $l\in A_1$ be a generic linear form. With $V: = (U:l)\subset A_{d-1}$ the following inequality holds
\[
\codim U^2 \le \codim UV.
\]
If furthermore $k\le d-1$, then
\[
\codim U^2 \le \codim V^2.
\]
\end{thm}
\begin{proof}
Since $l$ is generic and $k\le d$ it follows from \cref{lem:green_small_k} that $\codim V=\codim U$ and $\id{U,l}_d=A_d$. Furthermore, we have 
\[
A_{2d}=(\id{U,l}_d)^2\subset \id{U^2,l}_{2d},
\] 
hence $\codim (U^2:l)=\codim U^2$ by the exact sequence in \ref{lab:setup_green}. Since $UV\subset (U^2:l)$ we have
\[
\codim U^2=\codim\, (U^2:l)\le \codim UV.
\]
Now we do the same for $UV$ if $k\le d-1$. If we show that $\id{V,l}_{d-1}=A_{d-1}$, then
\[
A_{2d-1}=\id{U,l}_d\id{V,l}_{d-1}\subset \id{UV,l}_{2d-1}.
\]
Thus $\codim (UV:l)=\codim UV$ and $V^2\subset (UV:l)$ which means $\codim UV\le \codim V^2$.

It is left to show that $\id{V,l}_{d-1}=A_{d-1}$. This is equivalent to showing that ${\codim(V:l)}= \codim V$. Since $((U:l):l)=(U:l^2)$ this again is equivalent to showing that ${\codim (U:l^2)}=\codim V=\codim U$ or $\id{U,l^2}_d=A_d$. Since $l\in A_1$ is generic, we can also apply a generic \cc\ to $U$, hence assume that $\initial(U)=\gin(U)$ and $l=x_1$. Then
\[
\dim\,\id{U,x_1^2}_d=\dim\initial(\id{U,x_1^2})_d\ge \dim\, \id{\initial(U),x_1^2}_d.
\]
Here the first equality follows from the fact that any ideal and its initial ideal have the same Hilbert function, the second one is immediate since $\initial(\id{U,x_1^2})\supset\id{\initial(U),x_1^2}$.

It is therefore enough to show that $\id{\gin(U),x_1^2}_d=A_d$. Since $k\le d-1$ every monomial of degree $d$ not contained in $\gin(U)_d$ is divisible by $x_1^2$. But this means exactly that $\gin(U)_d+x_1^2A_{d-2}=A_d$.
\end{proof}

\begin{rem}
(i) The reason we pass to initial ideals in the second part of the proof is that we need to show $\id{UV,l}_{2d-1}=A_{2d-1}$.
As we have seen $\id{U,l}_d=A_d$ and if we take another generic linear form $l'$ we also have $\id{V,l'}_{d-1}=A_{d-1}$. However, since $V=(U:l)$ we do not know that $l$ behaves generically for $V$.

(ii) It is not true in general that $(U:l)$ is \bpf\ if $U$ is. Let $n=3$ and let $U=\spn(x^2y, x^2z, xy^2)^\perp\subset \C[x,y,z]_3$. Then $U$ contains $A_1\spn(yz,z^2)$, and thus for a generic linear form $l\in A_1$ we have $(U:l)=\spn(yz,z^2)\oplus\spn(p)$ for some $p\in A_2$. Hence, the space $(U:l)$ has a \bp, namely $\sz(z,p)$.

One can show however that $(U:l)$ is \bpf\ whenever the degree is large enough.
\end{rem}

Now we show the reverse inequality from \cref{thm:deg_reduction} in the case of strongly stable subspaces.

\begin{prop}
\label{prop:deg_reduction_equality_ss}
Let $U\subset A_d$ be a strongly stable subspace of codimension $k\le d-1$ and let $V: = (U:x_1)$. Then $\codim V^2\le\codim U^2$
\end{prop}
\begin{proof}
Let $M\in A_{2d-2}\setminus V^2$. We show that if $x_1^2M\notin U^2$, then $x_1^2((V^2)^\perp)\subset (U^2)^\perp$, and the claimed inequality follows.

Assume $x_1^2M\in U^2$. Then either 
\begin{enumerate}
\item there exist monomials $S,T\in A_{d-1}$ such that $x_1^2M=(x_1S)(x_1T)$ and $x_1S,\,x_1T\in U$, or
\item there exist $S\in A_{d-2}$ and $T\in A_d$ such that $x_1^2=(x_1^2S)T$ and $x_1^2S,\,T\in U$.
\end{enumerate}
In both cases $M=ST$. In case (i) we have $S,\,T\in V$ since $V=(U:x_1)$ and hence $M=ST\in V^2$, a contradiction.

In case (ii) we see $x_1S\in V$. If $x_1\vert T$, then $\frac{T}{x_1}\in V$ and again we have $M=(x_1S)\frac{T}{x_1}\in V^2$. Thus we can assume that $x_1$ does not divide $T$. Since $k\le d-1$ every monomial of degree $d-1$ not contained in $V$ is divisible by $x_1$. Hence, for every $i\in\{2,\dots,n\}$ such that $x_i\vert T$, the monomial $\frac{T}{x_i}$ is contained in $V$.
Fix any such $i\in\{2,\dots,n\}$. Since $V$ is strongly stable and $x_1S\in V$, the monomial $x_iS$ is also contained in $V$. Combined this gives
\[
M=ST=(x_iS)\frac{T}{x_i}\in V^2,
\]
which is again a contradiction.
\end{proof}

Combining the two inequalities of \cref{thm:deg_reduction} and \cref{prop:deg_reduction_equality_ss}, we get the following result.

\begin{cor}
\label{cor:ss_bound_independent_of_d}
If $k\le d$ then 
\[
m(n,d,k) = m(n,k,k).
\]
\end{cor}
\begin{proof}
If $d=k$ there is nothing to show, we can thus assume that $k\le d-1$.
Let $U\subset A_d$ be a subspace of codimension $k$ such that $\codim U^2=m(n,d,k)$ and $k<d$. 
By \cref{thm:deg_reduction} we have $\codim U^2\le \codim V^2$ with $V=(U:l)$ for a generic linear form $l\in A_1$. By definition $\codim V^2\le m(n,d-1,k)$, thus $m(n,d,k)\le m(n,d-1,k)$.
On the other hand, let $V\subset A_{d-1}$ be a strongly stable subspace of codimension $k$ such that $\codim V^2=m(n,d-1,k)$. 
Let $U: = x_1V\oplus \C[x_2,\dots,x_n]_d$, then $V=(U:x_1)$ and $U$ is strongly stable. By \cref{prop:deg_reduction_equality_ss} it follows that $m(n,d-1,k)=\codim V^2\le\codim U^2\le m(n,d,k)$. 
Combined this gives $m(n,d,k)=m(n,d-1,k)$ and we are done by induction.
\end{proof}

\begin{rem}
The proofs also show that if $V\subset A_d$ is a strongly stable subspace of codimension $k$ such that $\codim V^2=m(n,d,k)$ and $k\le d$, the subspace $U: = x_1V\oplus\C[x_2,\dots,x_n]_{d+1}$ satisfies $\codim U^2=m(n,d+1,k)$.
\end{rem}

\section{Lifting subspaces}
\label{sec:Lifting subspaces}

In \cref{thm:reduction_number_of_variables} we showed how to reduce the number of variables, now we also want to increase that number while preserving $\codim U^2$.

\begin{dfn}
Let $U\subset A_d$ be a subspace of codimension $k$. Define
\[
U^{(1)}: = x_{n+1}A(n+1)_{d-1}\oplus U\subset A(n+1)_d
\]
and for any $l\ge 2$
\[
U^{(l)}: = (U^{(l-1)})^{(1)}\subset A(n+l)_d
\]
($U^{(0)}: = U$). 
\end{dfn}

For any $l\ge 1$ the subspace $U^{(l)}$ also has codimension $k$ in $A(n+l)_d$. And in fact, we know the whole Hilbert function of $U^{(l)}$.

\begin{prop}
\label{prop:liftingfaces}
Let $U\subset A_d$ be a subspace of codimension $k$. Let $H=(h_i)_{i\ge 0}$ be the Hilbert function of $\id{U}$. Then for every $l \ge 0$ the following hold:
\begin{enumerate}
\item The Hilbert function $K=(k_i)_{i\ge 0}$ of the ideal generated by $U^{(l)}$ in $A(n+l)$ satisfies
\begin{itemize}
\item $k_i=\dim A(n+l)_i$ for $0 \le i \le d-1$ and
\item $k_i=h_i$ for $i\ge d$.
\end{itemize}
Furthermore, we have
\item $\codim_{A(n+l)_{2d}} (U^{(l)})^2 = \codim_{A_{2d}} U^2 + l\cdot h_{2d-1}$.
\end{enumerate}
\end{prop}
\begin{proof}
It is enough to show this for $l=1$ since the rest follows by induction. Write $A'=A[y]$ with a new indeterminate $y$, then
\[
V: = U^{(1)}=yA'_{d-1}\oplus U\subset A'_d.
\]
For any $s\ge 0$, we have
\begin{align*}
VA_{s}'&=yA_{d-1}'A_{s}'+UA_{s}'=\left(\bigoplus_{i=1}^{d+s} y^iA_{d+s-i}\right)+A_{s}U+yA_{s-1}U+\dots+y^sU\\
&=\bigoplus_{i=1}^{d+s} y^iA_{d+s-i} \oplus UA_{s}
\end{align*}
which shows (i) since $A'_{d+s}=\bigoplus_{i=0}^{d+s} y^iA_{d+s-i}$.

For (ii) we calculate $V^2$ and with the same argument as above we get
\[
V^2=y^2A_{2d-2}'+yA'_{d-1}U+U^2=\bigoplus_{i=2}^{2d} y^iA_{2d-i} \oplus y(A_{d-1}U) \oplus U^2
\]
and 
\[
\codim_{A'_{2d}} V^2 = \codim_{A_{2d}} U^2 + h_{2d-1}.
\]
\end{proof}

This enables us to determine the Hilbert function of codimension $2$ subspaces of $A_2$ as an easy application.

For generic $U$ the Hilbert function of $\id{U}$ is as small as possible. In the codimension 2 case this means that the Hilbert function is $(1,n,2)$ generically. We show that this holds whenever $U$ is \bpf.

\begin{prop}
\label{prop:quadratic_hf}
Let $U\subset A_2$ be a \bpf\ subspace of codimension 2. Then the Hilbert function of $\id{U}$ is $(1,n,2)$.
\end{prop}
\begin{proof}
By \cref{thm:macaulay} the Hilbert function is smaller or equal to $(1,n,2,2,\dots)$. So assume $h_{\id{U}}(3)>0$. Then by \cref{prop:liftingfaces}, the subspace $U^{(l)}\subset A(n+l)_2$ has codimension 2 and for $l \ge 7$ we have $\codim_{A(n+l)_4} (U^{(l)})^2\ge 7$ which is not possible by \cref{thm:codimension2bounds}.
\end{proof}

\begin{cor}
\label{cor:lift_codim_1_2}
Let $U\subset A_d$ be a \bpf\ subspace of codimension $k\in\{1,2\}$ and $\codim U^2=s$. Then for every $N\ge n$, there exists a \bpf\ subspace $V\subset A(N)_d$ of codimension $k$ such that $\codim V^2=s$.
\end{cor}
\begin{proof}
By \cref{prop:quadratic_hf} and \cref{cor:degree2d-1} the degree $2d-1$ component of $A/\id{U}$ has dimension $0$. Hence 
\[
\codim_{A(n)_{2d}} U^2=\codim_{A(n+l)_{2d}} (U^{(l)})^2.
\]
by \cref{prop:liftingfaces} (ii).
\end{proof}

\section{Arbitrary codimension}
\label{sec:arbitrary codimension}

We show bounds for $m^0(n,d,k)$ that are independent of $n$ and $d$, if $d$ is large enough.
The most important step is to also consider the orthogonal complement alongside our starting space. This is made precise in \cref{lem:dual}.

The main idea is the following: if $U\subset A_d$ is a \bpf\ subspace of codimension $k$, we consider $U': = U\cap A(m)_d$ for some $2\le m \le n$. To use \cref{thm:reduction_number_of_variables} we need to make sure that $\codim U=\codim U'$ and to get bounds that are independent of $n$ we want $U'$ to be \bpf\ as well (and $k\le d-1$), we then use \cref{cor:reduction_number_of_variables} to conclude $\codim U^2\le \codim (U')^2$. To get a bound that is independent of $d$, we use the results of \cref{sec:Reducing the degree}.

We still always assume that $n,\,d\ge 2$, and $k\in\N$. 

\begin{rem}[The dual problem]
\label{rem:dual}
Let $U\subset A_d$ be a \bpf\ subspace of codimension $k$. Instead of asking if $U': =  U\cap A(m)_d$ satisfies 
\begin{enumerate}
\item $\codim_{A(m)_d} U'=k$ and
\item $\sz(U')=\emptyset$ with $\sz(U')\subset\P^{m-1}$,
\end{enumerate}
as in \cref{thm:reduction_number_of_variables} and \cref{cor:reduction_number_of_variables}, we can also look at the dual problem:

Let $W=U^\perp$. Does $W': = W(x_1,\dots,x_m,0,\dots,0)$ have the same dimension as $W$ and does $W'$ not contain the $d$-th power of a linear form.
\end{rem}

The next lemma is an easy statement from linear algebra.

\begin{lem}
\label{lem:dual}
Let $U\subset A_d$ be a subspace and $W: = U^\perp$. Let $l_1,\dots,l_s\in A_1$ be linearly independent linear forms and $V: = \C[l_1,\dots,l_s]_d\subset A_d$. Then 
\[
(U\cap V)^\perp\cong(W+V^\perp)/V^\perp,
\]
and $V^\perp=\spn(\lambda_1,\dots,\lambda_{n-s})A_{d-1}$ where $\spn(l_1,\dots,l_s)^\perp=\spn(\lambda_1,\dots,\lambda_{n-s})\subset A_1$.

Write $\ol W$ for $(W+V^\perp)/V^\perp$, then we especially have $\codim U\cap V=\dim \ol W$ and $U\cap V$ is \bpf\ if and only if $\ol W$ contains no $d$-th power of a linear form.
\end{lem}

Now we want to work on condition (i) in \cref{rem:dual} to ensure that $\dim W=\dim \ol W$.
Since $W$ will play the role of $U^\perp$, $k$ will usually denote the dimension of $W$, and not the codimension.

\begin{prop}
\label{prop:green_application_dimW}
Let $k < n$ and let $W\subset A_d$ be a subspace of dimension $k$. Let $l\in A_1$ be generic. Then $\dim \ol W=\dim W$ where $\ol W=W+\id{l}_d/\id{l}_d$.
\end{prop}
\begin{proof}
Using the notation of \cref{lab:setup_green} with $I=\id{W}$, we have 
\[
h_I(d)=h_d=\dim A_d-k=\sum_{i=0}^{d-2} \underbrace{\bin{(n-2)+d-i}{d-i}}_{=\dim A(n-1)_{d-i}} + \bin{n-k}{1}.
\]
Green's \cref{thm:green_thm} then shows
\[
\dim A_d/\id{W,l}_d = c_d\le \underbrace{\sum_{i=0}^{d-2} \bin{(n-2)+d-i-1}{d-i}}_{C: = } + \bin{n-k-1}{1}.
\]
As can be easily verified we have $C=\dim A(n-1)_d-(n-1)$ and therefore
\[
c_d\le \dim A(n-1)_d-(n-1)+\bin{n-k-1}{1} = \dim A(n-1)_d - k.
\]
Equivalently $\dim \ol W\ge k$.
Since $\dim \ol W\le \dim W=k$, it follows that $\dim \ol W=k$.
\end{proof}

\begin{rem}
The condition of \cref{prop:green_application_dimW} on the dimension, namely $k < n$ is necessary and tight in the following sense. Let $0\neq F\in A_{d-1}$ and let $W=FA_1$. Then $\dim W=n$ and $\dim\ol W=n-1<\dim W$ for a generic linear form $l\in A_1$.

In fact, it follows from \cite[Theorem 3.2]{ams2018} that every subspace $W$ of dimension $n$ such that $\dim \ol W < \dim W$ for a generic linear form $l\in A_1$, has the form $FA_1$ for some $F\in A_{d-1}$.
\end{rem}

\begin{cor}
\label{cor:dim_intersection}
Let $U\subset A_d$ be a subspace of codimension $k$ with $k\le n$. Then 
\[
\codim(U\cap \C[l_1,\dots,l_k]_d)=k
\]
for generic linear forms $l_1,\dots,l_k\in A_1$.

Especially, after applying a generic \cc\ to $U$, we have 
\[
\codim (U\cap A(k)_d)=k.
\]
\end{cor}
\begin{proof}
Let $W=U^\perp$. If $k=n$ there is nothing to show, hence assume $k<n$. By \cref{lem:dual} it is enough to consider the dimension of $\ol W\subset A_d/\id{l_{k+1},\dots,l_n}_d$ for any basis $l_{k+1},\dots,l_n$ of the orthogonal complement of $\spn(l_1,\dots,l_k)$. Since $k<n$, it follows from \cref{prop:green_application_dimW} that $\dim W=\dim \ol W$.
\end{proof}

\begin{cor}
\label{cor:green_dim(W:l)}
Let $W\subset A_d$ be a subspace of dimension $k<n$ and let $l\in A_1$ a generic linear form. Then $\dim\, (W:l)=0$.
\end{cor}
\begin{proof}
This follows from \cref{prop:green_application_dimW} and the exact sequence 
\[
0 \to A_{d-1}/(W:l) \to A_d/W \to A_d/\id{W,l}_d \to 0.
\]
in \cref{lab:setup_green}: the space in the middle has dimension $\dim A_d-k$ and the space on the right has dimension $\dim A(n-1)-k$ by \cref{prop:green_application_dimW}.
Hence, the one on the left-hand side has dimension $\dim A_d-k-\dim A(n-1)+k=\dim A_{d-1}$ and thus $\dim (W:l)=0$.
\end{proof}

This shows that condition (i) in \cref{rem:dual} is satisfied whenever $k<n$ and we go down by one variable. And it is in general not satisfied if $n\le k$ since we can take $W=FA_1$ for some $0\neq F\in A_{d-1}$.

Now we want to look at condition (ii) in \cref{rem:dual}. By \cref{lem:dual} asking whether $U': = U\cap A(n-1)$ is \bpf\ is the same as asking if the orthogonal complement contains no $d$-th power of a linear form. Thus assume that $U$ is \bpf\ and $W$ contains no $d$-th powers. 

Is it true in general that $\ol W\subset (A/\id{l})_d$ contains no $d$-th powers for generic $l\in A_1$ whenever $\dim W=\dim \ol W$? Sadly this is not the case as the next example shows.

\begin{example}
Let $W: = x_n^{d-1}A(n-1)_1\subset A_d$ and let $l\in A_1$ be a generic linear form. After scaling $l$, we can write $l=x_n+l'$ for some $l'\in A(n-1)_1$, hence $\ol W\subset (A/\id{l})_d$ is isomorphic to $(l')^{d-1}A(n-1)_1$. Then $(l')^d\in \ol W$ and $\dim W=\dim \ol W$.
\end{example}

However, it is true whenever the number of variables is large as the next theorem shows.
For convenience, we use the following notation. For a subspace $W\subset A_d$ we say $l\in A_1$ is $W$-generic if $l$ is generic in the sense of Green's \cref{thm:green_thm}.

\begin{thm}
\label{prop:stay_bpf}
Let $W\subset A_d$ be a subspace of dimension $k$ and let $n\ge 2k+1$. If $W$ contains no $d$-th power of a linear form, then the same holds for $\ol W=W+\id{l}_d/\id{l}_d\subset (A/\id{l})_d$ where $l\in A_1$ is a generic linear form.
\end{thm}
\begin{proof}
Assume this is wrong.
Let $X=\vc(W)\subset\P^{n-1}$. Let $l\in A_1$ be generic, then $\vc(W,l)=\vc(W,l,L^d)$ for some $L\in A_1$ since $L^d+pl\in W$ for some $p\in A_{d-1}$.
Hence a generic hyperplane section of $X$ is degenerate. We claim that $X$ is degenerate. If any irreducible component of $X$ is non-degenerate, then so is a generic hyperplane section of this component and hence $X$. Therefore all irreducible components are degenerate. 
Assume $X$ is non-degenerate. Then there exist components $X_1,\dots,X_s$ each contained in a linear variety $T_i$ but the union of all $T_i$ is non-degenerate. Note that all $T_i$ have codimension at most $k$ and $n\ge 2k+1$, thus no $T_i$ has dimension 0. Intersecting $X$ with $\vc(l)$, each $X_i$ is a non-degenerate variety inside $T_i\cap\vc(l)$ and the union $\cup_{i=1}^s (T_i\cap\vc(l))$ is non-degenerate inside $\vc(l)$. But then $X\cap\vc(l)$ is non-degenerate in $\vc(l)$, a contradiction.

Hence there exists $H_1,\dots,H_r\in A_1$ linearly independent such that $X\subset \vc(H_1,\dots, H_r)$, $r\le k$ and $X$ is non-degenerate in $\vc(H_1,\dots,H_r)$.
Let $l_1,\dots,l_{k+1}\in A_1$ generic, then there exist $L_1,\dots,L_{k+1}\in\spn(H_1,\dots,H_r)$ such that $L_i^d+l_ig_i\in W$ for some $g_i\in A_{d-1}$ and $i=1,\dots,k+1$. Since $\dim W=k$ these $k+1$ forms are linearly dependent and there exist $\lambda_1,\dots,\lambda_{k+1}\in\C$ not all zero such that
\[
\sum_{i=1}^{k+1}\lambda_i (L_i^d+l_ig_i)=0.
\]
Since $l_1,\dots,l_{k+1}$ are generic and $r+k+1\le 2k+1$, the linear forms $H_1,\dots,H_r,l_1,\dots,l_{k+1}$ are linearly independent. After a \cc\ we may therefore assume that $H_i=x_i$ ($i=1,\dots,r$) and $l_i=x_{r+i}$ ($i=1,\dots,k+1$). Since
\[
\sum_{i=1}^{k+1}\lambda_i L_i^d +\sum_{i=1}^{k+1}\lambda_i l_ig_i=0.
\]
but no monomial in the second sum is contained in $\C[x_1,\dots,x_r]$ we see that both sums individually have to be 0. With $j:=\max(i\colon \lambda_i\neq 0)$ we have 
\[
l_jg_j\in\spn(l_1g_1,\dots,l_{j-1}g_{j-1}).
\]
Since $\dim\spn(l_1g_1,\dots,l_{j-1}g_{j-1})\le k$ it follows from Green's theorem that 
\[
(\spn(l_1g_1,\dots,l_{j-1}g_{j-1})\colon l)=\{0\}
\]
for generic $l\in A_1$. Especially, this is true for $l_j$, a contradiction.
\end{proof}

\begin{thm}
\label{thm:independent_bound}
Let $k\le d-1$. Then for every $n\ge 2$ and every \bpf\ subspace $U\subset A_d$ of codimension $k$ we have 
\[
\codim U^2\le m(2k,k,k).
\]
Especially, the number $m(2k,k,k)$ is independent of $n$ and $d$.
\end{thm}
\begin{proof}
If $n\le 2k$ and $U\subset A_d$ is a \bpf\ subspace of codimension $k$, then by \cref{prop:liftingfaces} we have $\codim U^2=\codim\, (U^{(2k-n+1)})^2$ and $\codim U=\codim\, U^{(2k-n+1)}$ with $U^{(2k-n+1)}\subset A(2k+1)_d$. It is therefore enough to only consider the case $n > 2k$.

We apply a generic \cc\ to $U$. By \cref{cor:dim_intersection} it follows that $V: = U\cap A(2k)_d$ has codimension $k$ in $A(2k)_d$ and by \cref{prop:stay_bpf} the subspace $V$ is still \bpf. 

Using \cref{cor:reduction_number_of_variables} we get
$\codim U^2\le \codim V^2 \le m(2k,d,k)$.
By \cref{cor:ss_bound_independent_of_d} we finally have $m(2k,d,k) = m(2k,k,k)$ which concludes the proof.
\end{proof}

\begin{rem}
Assuming that $d$ is large enough is essential. Consider the following example: Let $R=A(4), \mm=\id{x_5,\dots,x_n}\subset A(n), n\ge 5$ and
\[
U=\spn(x_1^3,x_2^3,x_3^3,x_4^3,x_1^2x_2+x_3^2x_4)\subset A(4)_3.
\]
This subspace has codimension 15 in $A(4)_3$. We check with {\ttfamily SAGE} \cite{sage} that the Hilbert function of $\id{U}$ is given by $(1,4,10,15,15,7,1)$. Define the subspace
\[
V: =U^{(n-4)}=\bigoplus_{i=1}^3 \mm^i R_{3-i}\oplus U\subset A(n)_3.
\]
This subspace also has codimension 15 in $A(n)$ by \cref{prop:liftingfaces}. Again by \cref{prop:liftingfaces} we know that $\codim V^2=\codim U^2+7\cdot (n-4)$.

This shows that we cannot have a uniform bound for this combination of codimension and degree not depending on $n$.
\end{rem}

It seems likely that one cannot only reduce to $2k$ variables in \cref{prop:stay_bpf} but actually to $k+1$ variables. This is at least the only counterexample we know of (for $k\le n-1$). We also checked this in small cases for all monomial subspaces on a computer.

\begin{conj}
\label{conj:bpf_intersection}
Let $k\le d-1,n-1$, and $n\ge 3$. Let $W\subset A_d$ be a subspace of dimension $k$ and suppose that $W$ contains no $d$-th power of a linear form. Then for a generic linear form $l\in A_1$ it holds that either
\begin{enumerate}
\item $\ol W$ contains no $d$-th power of a linear form, or
\item $n=k+1$ and $W=L_1^{d-1}\C[L_2,\dots,L_{k+1}]_1$ for some basis $L_1,\dots,L_{k+1}$ of $A(k+1)_1$.
\end{enumerate}
\end{conj}

This would allow us to show $\codim U^2\le m(k,k,k)$ in \cref{thm:independent_bound} with an additional argument.

\begin{rem}
(i) The conjecture is certainly false if $n=k$ is allowed: for $n=k$ let $W=x_1^{d-1}\C[x_2,\dots,x_n]_1\oplus \spn(p)$ for some generic $p\in A_d$. Since $p$ is generic, $W$ contains no $d$-th powers, but $\ol W\cong x_1^{d-1}\C[x_1,\dots,x_{n-1}]_1\oplus\spn(\ol p)$ does contain one where $\ol W\subset A_d/\id{l}_d$ for a generic linear form $l\in A_1$.

(ii) The conjecture is true for $n\ge 2k+1$ by \cref{prop:stay_bpf}. For $k=1$ it also follows from a simple geometric observation: if $W=\spn(p)$ and $p$ is not a power of a linear form, then $\sz(p)$ is non-degenerate. Hence, the same holds for a generic hyperplane section which therefore cannot be defined by the power of a linear form.
\end{rem}

\begin{rem}
By definition of $m(n,d,k)$ we are considering complex subspaces of $A_d$. However, when studying faces of \gsa\ we are only interested in real subspaces.
As we have seen $m(n,d,k)$ is realized by a monomial subspace, hence we can consider it as a real subspace of $\Rx_d$. So the bound is also tight in the real case.

In general, we do not know all attainable values for $\codim U^2$ while fixing $n,d$ and $k$. Even in the case of codimension 2 subspaces, we do not know whether some values are possible.
\end{rem}

\section{Determining the upper bound}
\label{sec:determine_bound}

In this section, we determine the value $m(2k,k,k)$. Moreover, we find the strongly stable subspace that realizes $m(k,k,k)$, and show that is it unique. We rely on \cref{prop:monomial_lift} and \cref{prop:subspace_monomial_reduction} to do most of the combinatorial work for the main results, which are \cref{thm:m-k-k-k} and \cref{thm:main_bound}.

We always assume $k\in\N,\, k\ge 2$.

\begin{lem}
\label{lem:2k_le_k}
For any $k\ge 2$ we have $m(2k,k,k)\le k^2+m(k,k,k)$.
\end{lem}
\begin{proof}
Let $U\subset A(2k)_k$ be some strongly stable subspace of codimension $k$ that realizes $m(2k,k,k)$. 
Since $\codim U=k$, every monomial not contained in $U$ is contained in $A(k)_k$. Especially, the codimension of $V:=U\cap A(k)_k$ is $k$. Using \cref{thm:reduction_number_of_variables} we see
\begin{align*}
\codim U^2 &\le \codim V^2 + (2k-k)\dim (A(k)/\id{V})_{2k-1}\\
&\le m(k,k,k) +k\dim (A(k)/\id{V})_{2k-1}.
\end{align*}
Since $U\subset A(k)_k$ has codimension $k$, we see from \cref{cor:macaulay_gotzmann} (i) that 
\[
\dim (A(k)/\id{V})_{2k-1}\le k
\]
which finishes the proof.
\end{proof}

In fact, it follows from \cref{rem:HF_ss_subspace} and \cref{rem:variable_reduction_tight} that we have equality in the last lemma.

To determine $m(2k,k,k)$ it now suffices to determine a bound for $m(k,k,k)$.
We show that the only strongly stable subspace realizing the bound $m(k,k,k)$ for $k\ge 2$ is $U=\spn(x_1^k,x_1^{k-1}x_2,\dots,x_1^{k-1}x_k)^\perp$ and $m(k,k,k)=\bin{k+2}{3}$.

First, we show that the subspace $U$ above does realize the bound $m(k,k,k)$.

\begin{lem}
\label{lem:lex_seg}
Let $U\subset A(k)_k$ be the subspace of codimension $k$ spanned by all monomials of degree $k$ except for $x_1^d,x_1^{d-1}x_2,\dots,x_1^{d-1}x_k$. Then $\codim U^2=\bin{k+2}{3}$.
\end{lem}
\begin{proof}
We easily see that the only monomials not contained in $U^2$ are 
\begin{itemize}
\item $x_1^{2k-1}x_i$ for $i=1,\dots,k$,
\item $x_1^{2k-2}x_ix_j$ for $i,j=2,\dots,k$,
\item $x_1^{2k-3}x_ix_jx_l$ for $i,j,l=2,\dots,k$
\end{itemize}
or equivalently $x_1^{2k-3}x_ix_jx_l$ for $i,j,l=1,\dots,n$. These are a total of $\frac{1}{6}(k^3+3k^2+2k)=\bin{k+2}{3}$ monomials.
\end{proof}

\begin{thm}
\label{thm:m-k-k-k}
For any $k\ge 2$ we have $m(k,k,k)=\bin{k+2}{3}$.
\end{thm}
\begin{proof}
We prove this by induction on $k$. For $k=2$ this is immediately checked.
Assume the claim holds for some $k\ge 2$. Let $U\subset A(k+1)_{k+1}$ be a strongly stable subspace realizing the bound $m(k+1,k+1,k+1)$. We write $d=k+1$ for the degree. If $x_1^{d-1}x_{k+1}$ is not contained in $U$, then $U=\spn(x_1^d,x_1^{d-1}x_2,\dots,x_1^{d-1}x_{k+1})^\perp$ by \cref{lem:p_determines_subspace}. By \cref{lem:lex_seg} we then have $\codim U^2=\bin{k+3}{3}=\bin{(k+1)+2}{3}$.
We show that this is the only strongly stable subspace realizing the upper bound. 
For the sake of contradiction assume that $x_1^{d-1}x_{k+1}$ is contained in $U$, then no monomial in $U^\perp$ is divisible by $x_{k+1}$. Indeed, if there was a monomial $M\in U^\perp$ divisible by $x_{k+1}$, we can write $M=T x_{k+1}$ with $T\in A(k+1)_{d-1}$, then $M=T\frac{x_1^{d-1}x_{k+1}}{x_1^{d-1}}\in U$ as $x_1^{d-1}x_{k+1}\in U$ and $U$ is strongly stable.
Therefore, we can write
\[
U=x_{k+1}A(k+1)_{d-1}\oplus V
\]
where $V\subset A(k)_d$ is a strongly stable subspace of codimension $k+1$. By \cref{thm:reduction_number_of_variables} we now have
\[
m(k+1,d,k+1)=\codim U^2\le \codim V^2 + k+1\le m(k,d,k+1)+k+1.
\]
In \cref{prop:subspace_monomial_reduction} we show that $m(k,d,k+1)< m(k,k,k)+\bin{k+1}{2}$, then we have
\[
m(k+1,k+1,k+1)=\codim U^2 < \underbrace{m(k,k,k)}_{=\bin{k+2}{3}}+\bin{k+1}{2}+k+1=\bin{k+3}{3}.
\]
However, the subspace $W: = \spn(x_1^d,x_1^{d-1}x_2,\dots,x_1^{d-1}x_{k+1})^\perp$ satisfies $\codim W^2=\bin{k+3}{3}$. This yields the following contradiction 
\[
\bin{k+3}{3}=\codim W^2\le m(k+1,k+1,k+1)< \bin{k+3}{3}.
\]
\end{proof}

It is left to prove \cref{prop:subspace_monomial_reduction} which we used in the last proof.
For this we first need some more preparation.

\begin{dfn}
For any monomial $M\in A_d$ denote by $\p(M)$ the smallest integer $j>1$ such that $x_j\vert M$. If $M=x_1^d$, we define $\p(M): = 1$. By $M^-$ we denote the \todfn{reduction of} $M$, which is the monomial $x_1\frac{M}{x_{\p(M)}}$. On the other hand, we write $M^+$ for the set of all monomials $T\in A_d$ such that $M$ is the reduction of $T$.
\end{dfn}

\begin{example}
Consider the monomial $M=x_1^2x_3x_4\in A(4)_4$, then $\p(M)=3$ and $M^-=x_1^3x_4$. The set $M^+$ consists of the monomials $x_1x_2x_3x_4$ and $x_1x_3^2x_4$. The monomial $x_1x_3x_4^2$ is not contained in $M^+$ since its reduction is $x_1^2x_4^2$.

If $M=x_1^d$, then by definition $M^-=x_1^d$ and $M^+=\{x_1^{d-1}x_i\colon 1\le i\le n\}$.
\end{example}

\begin{lem}
\label{lem:monomial_reduction}
Let $M\in A_d$ be a monomial divisible by $x_1$ with $\p(M)>1$. Then $\vert M^+\vert = \p(M)-1$.
\end{lem}
\begin{proof}
$M^+$ consists of the elements $x_j\frac{M}{x_1}$ for $j=2,\dots,\p(M)$: The variables $x_2,\dots,x_{\p(M)-1}$ do not appear in $M$ by definition of $\p(M)$. Hence, for $2\le j\le \p(M)$ we see $\p\left(x_j\frac{M}{x_1}\right)=j$, and therefore
$\left(x_j\frac{M}{x_1}\right)^-=\frac{x_1}{x_j}\left(x_j\frac{M}{x_1}\right)=M.$
\end{proof}

\begin{lem}
\label{lem:p_determines_subspace}
Let $U\subset A_d$ be a strongly stable subspace of codimension $k\le n$. If there exists a monomial $M\in U^\perp$ such that $\p(M)=k$, then $U=\spn(x_1^d,x_1^{d-1}x_2,\dots,x_1^{d-1}x_k)^\perp$.
\end{lem}
\begin{proof}
Since $M\notin U$ and $U$ is strongly stable, the monomials $x_{i}\frac{M}{x_k}$ are also not contained in $U$ for $i=1,\dots,k-1$. Hence, $U^\perp$ is spanned by the monomials $x_{i}\frac{M}{x_k}$ ($i=1,\dots,k-1$) and $M$. Especially, if $M=x_1^{d-1}x_k$, then $U$ has the asserted form.

Assume $M\neq x_1^{d-1}x_k$, hence $M=x_1^a x_k T$ for some $a\in\N$ and some monomial $1\neq T\in\C[x_k,\dots,x_n]$ of degree $d-a-1$. Just as above, the monomials $x_1^ax_i T$ are not contained in $U$ for $i=1,\dots,k$. However, since $T\neq 1$, there exists $k<i\le n$ such that $x_i\vert T$, hence $x_1^ax_k^2\frac{T}{x_i}$ is also not contained in $U$. This is a contradiction since $k+1$ different monomials are not contained in $U$.
\end{proof}

\begin{prop}
\label{prop:monomial_lift}
Let $U\subset A_d$ be a strongly stable subspace of codimension $k$ with $2\le k\le n$. Let $S$ be the set of all monomials of degree $d$ not contained in $U$. Then $|\bigcup_{M\in S} M^+| \le \bin{k}{2} + n$
with equality if and only if $U=\spn(x_1^d,x_1^{d-1}x_2,\dots,x_1^{d-1}x_k)^\perp$. Furthermore, every monomial in $S$ is contained in $\bigcup_{M\in S} M^+$.
\end{prop}

\begin{rem}
We give a brief idea why \cref{prop:monomial_lift} should hold. Assume we know that equality in \cref{prop:monomial_lift} holds whenever $U=\spn(x_1^d,x_1^{d-1}x_2,\dots,x_1^{d-1}x_k)^\perp$. We swap one monomial in $U^\perp$ with some other monomial in $A_d$. For $U$ to stay strongly stable, we need to remove $x_1^{d-1}x_k$ and add $x_1^{d-2}x_2^2$. By \cref{lem:monomial_reduction} we know that $|(x_1^{d-2}x_2^2)^+|=1$ and $|(x_1^{d-1}x_k)^+|=k-1$, hence the inequality is now strict.

Continuing doing this, we have to remove $x_1^{d-1}x_j$ for some large $j$ and add a monomial $M$ in fewer variables, especially $\p(M)< \p(x_1^{d-1}x_j)=j$ and by \cref{lem:monomial_reduction} the set $\bigcup_{M\in S} M^+$ contains even fewer elements now.
\end{rem}

\begin{proof}[Proof of \cref{prop:monomial_lift}.]
Let $T\in S$ such that $\p(T)=\max\{\p(M)\colon M\in S\}$. Since $U$ is strongly stable, the elements $x_j\frac{T}{x_{\p(T)}}$ for $j=2,\dots,\p(T)-1$ are also not contained in $U$ and therefore lie in $S$. Moreover, they satisfy $\p(x_j\frac{T}{x_{\p(T)}})=j$ for all $j=2,\dots,\p(T)-1$.
Thus we have $\p(T)-1$ elements in $S$ for which we know the value $\p(\cdot)$, namely $T,\,x_j\frac{T}{x_{\p(T)}}$ ($j=2,\dots,\p(T)-1$).

Next, we look at the element $x_1^d\in S$. We have $(x_1^d)^+=\{x_1^d,\dots,x_1^{d-1}x_n\}$, and therefore $\vert (x_1^d)^+ \vert=n$.
In total, we identified $\p(T)$ elements in $S$ for which we know the value $\p(\cdot)$ and all other $k-\p(T)$ elements $M$ in $S$ satisfy $\p(M)\le \p(T)$ by the choice of $T$. 
We write
\[
S'=S\setminus\{x_1^d,x_2\frac{T}{x_{\p(T)}},\dots,x_{\p(T)-1}\frac{T}{x_{\p(T)}},T\}
\]
for the set where we removed all monomials from $S$ of which we determined the value $p(\cdot)$. We calculate
\begin{align*}
\bigg|\bigcup_{M\in S} M^+\bigg| &\le \sum_{i=2}^{\p(T)} \bigg|\left(x_i\frac{T}{x_{\p(T)}} \right)^+\bigg| + \vert (x_1^d)^+ \vert + \bigg|\bigcup_{M\in S'} M^+\bigg|\\
\text{\tiny{(\cref{lem:monomial_reduction})}}&\le \sum_{i=2}^{\p(T)} (i-1) + n + \underbrace{(\vert S\vert - \p(T))}_{=\vert S'\vert} (\p(T)-1)\\
&=\sum_{i=1}^{\p(T)-1} i + (k - \p(T)) (\p(T)-1)+n\\
&=\sum_{i=1}^{\p(T)-1} i + \sum_{j=\p(T)}^{k-1} (\p(T)-1)+n\le \sum_{i=1}^{k-1} i +n=\bin{k}{2}+n.
\end{align*}

Next, we show that equality holds if and only if $U=\spn(x_1^d,x_1^{d-1}x_2,\dots,x_1^{d-1}x_k)^\perp$.
First, we note that $\p(T)=k$ implies $S=\{x_1^d,x_1^{d-1}x_2,\dots,x_1^{d-1}x_k\}$ by \cref{lem:p_determines_subspace}.

Assume $U$ is not of this form, then $\p(T)<k$ and therefore the sum $\sum_{j=\p(T)}^{k-1} (\p(T)-1)$ is non-zero. Moreover, we have a strict inequality 
$\sum_{j=\p(T)}^{k-1} (\p(T)-1) < \sum_{j=\p(T)}^{k-1} j$
which means the second to last inequality above is also strict in this case. 

Therefore, equality can only hold if $U=\spn(x_1^d,x_1^{d-1}x_2,\dots,x_1^{d-1}x_k)^\perp$.
Assume we are in this case, then $\p(T)=k$. First, we note that this means that $S'=\emptyset$, hence we get
\[
\bigg|\bigcup_{M\in S} M^+\bigg| \le \sum_{i=2}^{\p(T)} \bigg|\left(x_i\frac{T}{x_{\p(T)}} \right)^+\bigg| + \vert (x_1^d)^+ \vert \stackrel{\text{\tiny{(\cref{lem:monomial_reduction})}}}{=} \sum_{i=2}^{\p(T)} (i-1) + n=\sum_{i=1}^{k-1} i +n= \bin{k}{2}+n.
\]
We thus need to show that the inequality is an equality or equivalently that the sets $M^+$ with $M\in S$ are disjoint.

Let $x_1^{d-1}x_i\in S,\, i\ge 2$, then $(x_1^{d-1}x_i)^+=\{x_1^{d-2}x_ix_l\colon 2\le l\le i\}$. 
We see that if $(x_1^{d-1}x_i)^+$ and $(x_1^{d-1}x_j)^+,\, i,j\ge 2$ contain a common element, it has the form $x_1^{d-2}x_ix_j$ with $i\le j$ and $j\le i$, which means $i=j$. 
For $x_1^d$ we have $(x_1^d)^+=\{x_1^d,\dots,x_1^{d-1}x_n\}$ and every element has degree at least $d-1$ in $x_1$, hence it is not contained in $(x_1^{d-1}x_i)^+$ for any $i\ge 2$. This finishes the first part.

The last claim we need to prove is $S\subset \bigcup_{M\in S} M^+$. Let $T\in S$. Since $U$ is strongly stable, the element $x_1\frac{T}{x_{\p(T)}}=T^-$ is also contained in $S$. But by definition this means 
\[
T\in \left(x_1\frac{T}{x_{\p(T)}}\right)^+\subset \bigcup_{M\in S} M^+
\]
which finishes the proof.
\end{proof}

\begin{lem}
\label{lem:bigger_ss_subspace}
Let $U\subset A_d$ be a strongly stable subspace of codimension $k+1$, then there exists a strongly stable subspace $V\subset A_d$ of codimension $k$ such that $U\subset V$.
\end{lem}
\begin{proof}
If $U=\{0\}$, take $V=\spn(x_n^d)$. Now assume $U\neq \{0\}$.
Let $M$ be any monomial not contained in $U$, and let $I:=\{i\in\N\colon x_i\vert M\}\setminus\{x_n\}$. Then either $x_{i+1}\frac{M}{x_i}\in U$ for every $i\in I$ or there exists $j\in I$ such that $x_{j+1}\frac{M}{x_j}\notin U$. 
If all of them are contained in $U$, we may add $M$ to $U$ and the subspace $V: = U\oplus\spn(M)$ is still strongly stable. Indeed, let $i\in I$ and $i<l\le n$, then
\[
x_l\frac{M}{x_i}=\frac{x_l}{x_{i+1}}\underbrace{\left(x_{i+1}\frac{M}{x_i}\right)}_{\in U}\in U.
\]
Any other monomial in $V$ except for $M$ is already contained in $U$ and $U$ is strongly stable.

If $x_{j+1}\frac{M}{x_j}$ is not contained in $U$ for some $j\in I$, then for any monomial ordering $\succeq$ we have $x_{j+1}\frac{M}{x_j}\succeq M$. 
Now we continue with $x_{j+1}\frac{M}{x_j}$ instead of $M$. After finitely many steps we find $M$ such that $x_{i+1}\frac{M}{x_i}\in U$ for every $i\in I$, since we reach $x_{n-1}x_n^{d-1}$ and this monomial is only divisible by $x_{n-1}$ and $x_n$ and 
$x_n\frac{x_{n-1}x_n^{d-1}}{x_{n-1}}=x_n^d\in U\neq\{0\}$.
\end{proof}

\begin{prop}
\label{prop:subspace_monomial_reduction}
For every $k\ge 2$ we have $m(k,k+1,k+1) < m(k,k,k)+\bin{k+1}{2}$.
\end{prop}
\begin{proof}
It is enough to show $m(k,k+1,k+1) < m(k,k+1,k)+\bin{k+1}{2}$ since by \cref{cor:ss_bound_independent_of_d} we have $m(k,k+1,k) = m(k,k,k)$. 
Again we denote the degree by $d: = k+1$. Let $U\subset A(k)_d$ be a strongly stable subspace of codimension $k+1$ realizing $m(k,d,k+1)$ and let $V\subset A(k)_d$ be any strongly stable subspace of codimension $k$ containing $U$. Such a subspace $V$ always exists by \cref{lem:bigger_ss_subspace}.

Now we compare $U^2$ and $V^2$. We can write $V=U\oplus\spn(M)$ for some monomial $M\in A(k)_d$ and hence $V^2=U^2+\spn(M)V$.

Let $MT\in\spn(M)V$ for some monomial $T\in V$. We claim that 
\[
MT=\left(x_{\p(T)}\frac{M}{x_1}\right)T^-
\]
is contained in $U^2$ except for at most $\bin{k+1}{2}-1$ choices of $T$. 
We note that the monomial $x_{\p(T)}\frac{M}{x_1}$ is always contained in $U$ since $U$ is strongly stable. After we show this, we are done as follows:
\[
\dim V^2\le \dim U^2 +\bin{k+1}{2}-1,\quad \text{ or equivalently }\quad \codim V^2\ge \codim U^2 - \bin{k+1}{2}+1
\]
hence
\[
m(k,k+1,k+1)=\codim U^2\le \codim V^2+\bin{k+1}{2}-1\le m(k,k+1,k)+\bin{k+1}{2}-1.
\]
Whenever $T^-$ is not contained in $U$ this means that $T$ is contained in $M^+$ for some monomial $M\in A(k)_d\setminus U$. 
The subspace $U$ has codimension $k+1$ and is contained in $A(k)_d$. Therefore, it cannot be the orthogonal complement of $\spn(x_1^d,x_1^{d-1}x_2,\dots,x_1^{d-1}x_{k+1})$ in $A(k+1)_d$ as the variable $x_{k+1}$ appears. 

Let $S$ be the set of all monomials in $A(k)_d$ that are not contained in $U$. It follows from \cref{prop:monomial_lift} that $\vert\bigcup_{M\in S} M^+\vert < \bin{k+1}{2} + k$. This means that there are at most $\bin{k+1}{2} + k -1$ monomials $M\in A(k)_d$ such that $M^-\notin U$.

Since $U\subset V$ and $V$ is strongly stable, all monomials not contained in $V$ are also not contained $U$, hence are in $S$ and therefore also in $\bigcup_{M\in S} M^+$ by \cref{prop:monomial_lift}.

We can now finish the proof: Let $T\in V$ such that $T^-\notin U$. This means 
\[
T\in \left(\bigcup_{M\in S} M^+\right) \cap V=\left(\bigcup_{M\in S} M^+\right)\setminus V^\perp
\]
and we just showed that the set on the right-hand side has cardinality at most 
$(\bin{k+1}{2}+k-1)-k=\bin{k+1}{2}-1.$
\end{proof}

\begin{thm}
\label{thm:main_bound}
Let $k\le d-1$. Then for every $n\ge 2$ and every \bpf\ subspace $U\subset A(n)_d$ of codimension $k$ we have 
\[
\codim U^2\le k^2+\bin{k+2}{3} = \frac{1}{6}(k^3+9k^2+2k).
\]
\end{thm}
\begin{proof}
From \cref{thm:independent_bound} we know $\codim U^2\le m(2k,k,k)$. From \cref{lem:2k_le_k} we see $m(2k,k,k)\le k^2+m(k,k,k)$ and \cref{thm:m-k-k-k} shows $m(k,k,k)=\bin{k+2}{3}$.
\end{proof}

\begin{thm}
Let $k\ge 1$ and let $n,d\ge k$. Then $m(n,d,k)=\bin{k+2}{3}+(n-k)k$. Furthermore the only strongly stable subspace $U\subset A_d$ realizing $m(n,d,k)$ is given by 
\[
U=\spn(x_1^d,x_1^{d-1}x_2,\dots,x_1^{d-1}x_k)^\perp\subset A_d.
\]
\end{thm}
\begin{proof}
For $k=1$, this follows from \cref{lem:basepoint} and the fact that there only exists a single strongly stable subspace of codimension 1. For $k\ge 2$ we calculate
\begin{align*}
m(n,d,k)&  \overset{(\ref{thm:reduction_number_of_variables})}{=} m(k,d,k)+(n-k)k\overset{(\ref{cor:ss_bound_independent_of_d})}{=} m(k,k,k)+(n-k)k\\
&\overset{(\ref{thm:m-k-k-k})}{=} \bin{k+2}{3}+(n-k)k = \frac{1}{6}(k^3-3k^2+2k)+nk.
\end{align*}
For the first equality, we also use \cref{rem:variable_reduction_tight} and the fact that the Hilbert function of the ideal generated by the subspace $U=\spn(x_1^k,x_1^{k-1}x_2,\dots,x_1^{k-1}x_k)^\perp\subset A(k)_k$ is $k$ for any degree at least $k$: from \cref{cor:macaulay_gotzmann} we see that the Hilbert function is at most $k$ in those degrees, but we immediately check that $x_1^d,x_1^{d-1}x_2,\dots,x_1^{d-1}x_k$ are not contained in $UA_{d-k}$.

Let $U\subset A_d$ be a strongly stable subspace of codimension $k$ realizing $m(n,d,k)$. Every monomial not contained in $U$ is contained in $A(k)_d$, hence $\codim V=k$ where $V:=U\cap A(k)_d$. Now we have
\begin{align*}
m(k,k,k)+(n-k)k& \overset{(\ref{thm:m-k-k-k})}{=}m(n,d,k)=\codim U^2\overset{(\ref{thm:reduction_number_of_variables})}{\le} \codim V^2 + (n-k)k \\
&\overset{\phantom{(3.5.5)}}{\le} m(k,d,k)+(n-k)k\overset{(\ref{cor:ss_bound_independent_of_d})}{=}m(k,k,k)+(n-k)k,
\end{align*}
and therefore both inequalities are equalities and we have $\codim V^2=m(k,k,k)$. We note that $V\subset A(k)_d$, and $d$ might still be larger than $k$. 
By the last part of the proof of \cref{thm:deg_reduction} we have $m(k,k,k)=\codim V^2\le \codim (V:x_1^{d-k})^2 \le m(k,k,k)$. Especially, $(V:x_1^{d-k})\subset A(k)_k$ is a subspace of codimension $k$ realizing $m(k,k,k)$ and thus it is equal to 
$\spn(x_1^k,\dots,x_1^{k-1}x_k)^\perp\subset A(k)_k$. 
This shows $V^\perp =x_1^{d-k}\spn(x_1^k,\dots,x_1^{k-1}x_k)\subset A(k)_d$ and thus also $U=\spn(x_1^d,\dots,x_1^{d-1}x_k)^\perp\subset A_d$.
\end{proof}

\section{Application to Gram spectrahedra}
\label{sec:gram_spec}

After showing results about squares of subspaces in the last sections, we now want to interpret their meaning for \gsa.

Let $C\subset\R^n$ be a convex set and let $F\subset C$ also be convex. The set $F$ is called a \todfn{face} of $C$ if $x,y\in C$, $x+y\in F$ already implies $x,y\in F$. There are two trivial faces, the empty set and the whole set $C$. Every other face is part of the (euclidean) boundary of $C$.
The dimension of a convex set $C$, $\dim C$, is defined as the dimension of the smallest affine linear subspace containing $C$.

Firstly, we can use our results about subspaces of codimension 1 to show an upper bound for the dimension of faces of \gsa. We write $\Sigma_{n,2d}\subset\Rx_{2d}$ for the set of all sums of squares, i.e. all $f\in\Rx_{2d}$ such that there exist $f_1,\dots,f_r\in\Rx_d$ with $f=\sum_{i=1}^r f_i^2$.

\begin{prop}
Let $f\in\Sigma_{n,2d}$, and $F\subset\gram(f)$ a face. Let $\theta$ be a relative interior point of $F$ and $U = \im\theta$. If $U\neq \Rx_d$, then 
\[
\dim F\le \bin{\dim \Rx_d}{2}-\dim \Rx_{2d}+n.
\]
\end{prop}
\begin{proof}
By \cref{prop:ss_is_minimal} there exists a strongly stable subspace $V\subset \Rx_d$ such that $\codim U^2 \le \codim V^2$ and $\dim U=\dim V$. Since $V$ is strongly stable, $V$ has $(1:0:\dots:0)$ as a \bp.
\[
\dim F\stackrel{\cref{eq:dim_face}}{=}\bin{\dim(U)+1}{2}-\dim (U^2)\le \bin{\dim \Rx_d}{2}-\dim(V^2)=\bin{\dim \Rx_d}{2}-(\dim \Rx_d-n).
\]
\end{proof}

\begin{rem}
In fact, one can show that equality in the last proposition holds if and only if $f\in\partial P_{n,2d}$ is on the boundary of the cone of non-negative polynomials with exactly one real zero and $F=\gram(f)$ or if $n=2$ and $\codim U=1$.
\end{rem}

We now continue discussing the bound $m(n,d,k)$ and what we showed in the last sections. First, we get the following from \cref{thm:main_bound}.

\begin{thm}
\label{thm:main_bound_gsa}
Let $f\in\Sigma_{n,2d}$ be \nonsing. If $F\subset\gram(f)$ is a face of corank $k$ with $1\le k\le d-1$, then
\[
\dim F\le \bin{\dim \Rx_d-k+1}{2}-\dim\Rx_{2d}+k^2+\bin{k+2}{3}.
\]
\end{thm}

Next, we compare this bound to the bound we get from $m(n,d,k)$. I.e. we compare \gsa\ of singular forms to \gsa\ of \nonsing\ forms. We always assume $1\le k\le d-1$.

\begin{rem}[singular form]
Let $U\subset\Rx_d$ be a strongly stable subspace of codimension $k$ realizing $m(n,d,k)$ and let $f\in\interior \Sigma U^2\subset U^2$. Since $U$ has a real \bp, the form $f$ lies on the boundary of the psd cone $P_{n,2d}$. The \gs\ $\gram(f)$ has a face corresponding to the subspace $U$ since $f\in\interior\Sigma U^2$. Especially, this face $F$ has dimension 
\begin{align*}
\dim F&=\bin{\dim \Rx_d-k+1}{2}-\dim \Rx_{2d} + m(n,d,k)\\
{\scriptstyle(\cref{rem:asymptotic_behavior})}&\ge\bin{\dim \Rx_d-k+1}{2}-\dim \Rx_{2d}+ kn.
\end{align*}
\end{rem}

\begin{rem}[\nonsing\ form]
Let $f\in\Sigma_{n,2d}$ be a \nonsing\ form and let $F\subset\gram(f)$ be a face with corresponding subspace $U$ of codimension $k$. Using \cref{thm:main_bound_gsa} we have an upper bound for the dimension of $F$,
\[
\dim F =\bin{\dim \Rx_d-k+1}{2}-\dim \Rx_{2d} + \codim U^2\le \bin{\dim \Rx_d-k+1}{2}-\dim \Rx_{2d} + k^2+\bin{k+2}{3}.
\]
We note that although the upper bound for the codimension of $U^2$ is independent of $n$, the dimension of $F$ certainly is not.
\end{rem}

\begin{rem}
We see that for large enough $n$, the dimensional differences between faces of \gsa\ of singular and \nonsing\ forms are arbitrarily large.

Moreover for $n$ large and any $k < d$, \emph{every} face of corank $k$ of the \gs\ of the singular form has a higher dimension than any face of corank $k$ of the \gs\ of the \nonsing\ form. Indeed, the codimension of $U^2$ for any \bpf\ subspace is bounded by $k^2+\bin{k+2}{3}$, whereas $\codim U^2\ge kn$ if $U$ has a \bp.
\end{rem}

\begin{rem}
(i) We note that understanding powers of subspaces is also interesting in itself since the coordinate ring of the image of the rational map 
\[
\P^{n-1} \dashrightarrow \P^r,\quad x\mapsto (f_1(x):\dots : f_r(x))
\]
is isomorphic to $\bigoplus_{i=0}^\infty U^i$ if $f_1,\dots,f_r\in A_d$ and $U=\spn(f_1,\dots,f_r)$ (cf. \cite{bc2018}).

From this point of view we only considered the degree 2 component of this ring and saw that the dimension of $U^2$ is minimal if and only if the map is not a morphism (if $r\ge \dim A_d-d+1$ and $n$ is large enough).

(ii) Similarly, understanding the Hilbert functions of \bpf\ subspaces is also an interesting topic. Here we give Blekherman's paper \cite{blekherman2015} as a reference where also several applications to optimization are given. He studies the conditions under which a \bpf\ subspace $U\subset A_d$ satisfies $UA_d\neq A_{2d}$ (see \cref{thm:degree2d}). 

(iii) We mention again the EGH conjecture (see \cref{rem:egh}) where one is interested in understanding Hilbert functions of ideals containing regular sequences using monomial ideals. Possible Hilbert function of homogeneous ideals $I\subset A$ have been completely characterized by Macaulay. 
However, adding the condition in \cref{rem:egh} that the monomial ideal contains a regular sequence in the same degrees, the problem becomes much harder.

This is very similar to our situation, as finding bounds for $\codim U^2$ is rather easy if we allow subspaces to have \bp s, but becomes more difficult if we require $U$ to be \bpf.
\end{rem}

\bibliographystyle{plain}

\begin{thebibliography}{10}

\bibitem{ams2018}
Jeaman {Ahn}, Juan~C. {Migliore}, and Yong-Su {Shin}.
\newblock {Green's theorem and Gorenstein sequences.}
\newblock {\em {J. Pure Appl. Algebra}}, 222(2):387--413, 2018.

\bibitem{blekherman2015}
Grigoriy {Blekherman}.
\newblock {Positive Gorenstein ideals}.
\newblock {\em {Proc. Am. Math. Soc.}}, 143(1):69--86, 2015.

\bibitem{bc2018}
Mats {Boij} and Aldo {Conca}.
\newblock {On Fr\"oberg-Macaulay conjectures for algebras.}
\newblock {\em {Rend. Ist. Mat. Univ. Trieste}}, 50:139--147, 2018.

\bibitem{bh1998}
Winfried Bruns and H.~J\"urgen Herzog.
\newblock {\em Cohen-Macaulay Rings}.
\newblock Cambridge Studies in Advanced Mathematics. Cambridge University
  Press, 2 edition, 1998.

\bibitem{cm2008}
Giulio {Caviglia} and Diane {Maclagan}.
\newblock {Some cases of the Eisenbud-Green-Harris conjecture}.
\newblock {\em {Math. Res. Lett.}}, 15(2-3):427--433, 2008.

\bibitem{eisenbud1995}
David Eisenbud.
\newblock {\em Commutative Algebra: with a View Toward Algebraic Geometry}.
\newblock Springer New York, 1995.

\bibitem{egh1993}
David {Eisenbud}, Mark {Green}, and Joe {Harris}.
\newblock {Higher Castelnuovo theory}.
\newblock In {\em {Journ\'ees de g\'eom\'etrie alg\'ebrique d'Orsay, France,
  juillet 20-26, 1992}}, pages 187--202. Paris: Soci\'et\'e Math\'ematique de
  France, 1993.

\bibitem{galligo1974}
Andre {Galligo}.
\newblock {A propos du th\'eor\`eme de preparation de Weierstrass}.
\newblock {Fonctions de plusieurs Variables complexes, Sem. Francois Norguet,
  Oct. 1970 - Dec. 1973, Lect. Notes Math. 409, 543-579 (1974).}, 1974.

\bibitem{green1989}
Mark Green.
\newblock Restrictions of linear series to hyperplanes, and some results of
  macaulay and gotzmann.
\newblock In {\em Algebraic Curves and Projective Geometry}, pages 76--86.
  Springer Berlin Heidelberg, 1989.

\bibitem{hartshorne1966}
Robin Hartshorne.
\newblock Connectedness of the hilbert scheme.
\newblock {\em Publications Math\'ematiques de l'IH\'ES}, 29:5--48, 1966.

\bibitem{ikl1999}
Anthony Iarrobino and Steve~L. Kleiman.
\newblock The gotzmann theorems and the hilbert scheme.
\newblock In {\em Power Sums, Gorenstein Algebras, and Determinantal Loci},
  pages 289--312. Springer Berlin Heidelberg, 1999.

\bibitem{macaulay1927}
F.~S. {Macaulay}.
\newblock {Some properties of enumeration in the theory of modular systems}.
\newblock {\em {Proc. Lond. Math. Soc. (2)}}, 26:531--555, 1927.

\bibitem{scheiderer2018}
C.~Scheiderer.
\newblock Extreme points of gram spectrahedra of binary forms.
\newblock {\em Discrete Comput Geom}, 67:1174--1190, 2022.

\bibitem{sage}
{The Sage Developers}.
\newblock {\em {S}ageMath, the {S}age {M}athematics {S}oftware {S}ystem
  ({V}ersion 8.8)}, 2019.
\newblock {\ttfamily https://www.sagemath.org}.

\end{thebibliography}

\end{document}